\documentclass[10pt,reqno]{amsart}
\usepackage{amssymb,mathrsfs,color}

\usepackage{enumerate}

\usepackage{graphicx}
\usepackage{xcolor} % A package to add color.
\usepackage{tensor}
\usepackage{slashed}

\usepackage{enumitem}
\definecolor{green}{rgb}{1,0.5,0} % Redefines the color green.
\newcommand{\nrm}[1]{\Vert#1\Vert}
\newcommand{\brk}[1]{\langle#1\rangle}
\newcommand{\set}[1]{\{#1\}}

\newcommand{\aleq}{\lesssim}

\newcommand{\lap}{\Dlt}

\newcommand{\ud}{\mathrm{d}}
\newcommand{\rd}{\partial}
\newcommand{\nb}{\nabla}

\newcommand{\imp}{\Rightarrow}

\newcommand{\bb}{\Big}

\newcommand{\gmm}{\gamma}

\newcommand{\dlt}{\delta}
\newcommand{\Dlt}{\Delta}
\newcommand{\eps}{\epsilon}

\newcommand{\lmb}{\lambda}
\newcommand{\Lmb}{\Lambda}

\newcommand{\omg}{\omega}

\newcommand{\bbH}{\mathbb H}

\newcommand{\bbR}{\mathbb R}

\newcommand{\calB}{\mathcal B}
\newcommand{\calC}{\mathcal C}

\newcommand{\calE}{\mathcal E}

\newcommand{\calH}{\mathcal H}

\newcommand{\calN}{\mathcal N}

\newcommand{\calS}{\mathcal S}
\newcommand{\calT}{\mathcal T}

\newcommand{\arctanh}{\mathrm{arctanh} \, }

\usepackage{wrapfig}
\usepackage{tikz}
\usetikzlibrary{arrows,calc,decorations.pathreplacing}
\definecolor{light-gray1}{gray}{0.90}
\definecolor{light-gray2}{gray}{0.80}

\definecolor{deepgreen}{cmyk}{1,0,1,0.5}

\newcommand{\E}{\mathcal{E}}

\newcommand{\HH}{\mathcal{H}}

\newcommand{\NN}{\mathcal{N}}

\newcommand{\CC}{\mathscr{C}}

\newcommand{\Hp}{\mathbb{H}}

\newcommand{\R}{\mathbb{R}}
\newcommand{\Sp}{\mathbb{S}}

\newcommand{\h}{\mathbf{h}}
\newcommand{\m}{\mathbf{m}}

\newcommand{\al}{\alpha}
\newcommand{\be}{\beta}

\newcommand{\e}{\varepsilon}
\newcommand{\fy}{\varphi}
\newcommand{\om}{\omega}
\newcommand{\la}{\lambda}

\newcommand{\De}{\Delta}

\newcommand{\Ga}{\Gamma}

\newcommand{\loc}{\operatorname{loc}}

\makeatletter

\newcommand{\Rmnum}[1]{\expandafter\@slowromancap\romannumeral #1@}
\makeatother

\newcommand{\I}{\infty}

\newcommand{\abs}[1]{\left\lvert{#1}\right\rvert}

\newcommand{\ali}[1]{\begin{align}\begin{split} #1 \end{split}\end{align}}
\newcommand{\ant}[1]{\begin{align*}\begin{split} #1 \end{split}\end{align*}}
\newcommand{\EQ}[1]{\begin{equation}\begin{split} #1 \end{split}\end{equation}}

\newcommand{\Del}[1]{}

\numberwithin{equation}{section}

\newtheorem{thm}{Theorem}[section]

\newtheorem{lem}[thm]{Lemma}
\newtheorem{prop}[thm]{Proposition}

\newtheorem{conj}{Conjecture}

\theoremstyle{remark}
\newtheorem{rem}{Remark}
\newtheorem{defn}{Definition}

\newcommand{\mand}{{\ \ \text{and} \ \  }}

\newcommand{\mas}{{\ \ \text{as} \ \ }}

\newcommand{\euc}{\mathrm{euc}}
\newcommand{\hyp}{\mathrm{hyp}}
\newcommand{\dvol}{\mathrm{dVol}}

\def\glei{\mathrm{eq}}

\newcommand{\nl}{\mathrm{nl}}

\usepackage{cite}

\begin{document}

\title[Equivariant wave maps on the hyperbolic plane with large energy]{Equivariant wave maps on the hyperbolic plane with large energy}

\author{Andrew Lawrie}
\author{Sung-Jin Oh}
\author{Sohrab Shahshahani}

\begin{abstract} 
In this paper we continue the analysis of equivariant wave maps from $2$-dimensional hyperbolic space $\Hp^2$ into surfaces of revolution $\NN$ that was initiated in~\cite{LOS1, LOS4}.  When the target  $\NN=\Hp^2$  we proved in~\cite{LOS1} the existence and asymptotic stability of a $1$-parameter family of finite energy harmonic maps indexed by how far each map wraps around the target. Here we conjecture that each of these harmonic maps is globally asymptotically stable, meaning that the evolution of any arbitrarily large finite energy perturbation of a harmonic map asymptotically resolves into the harmonic map itself plus free radiation. Since such initial data exhaust the energy space, this is the~\emph{soliton resolution conjecture} for this equation. The main result is a verification of this conjecture for a nonperturbative subset of the harmonic maps.  %In~\cite{LOS1, LOS4} we investigated the stability properties of the $1$-parameter families of finite energy harmonic maps in the model cases $\NN = \Hp^2$ and  $\NN = \Sp^2$. %When the target $\NN = \Hp^2$ the $1$-parameter family of station 

%Here we study the asymptotic behavior of solutions with large energy and prove that the~\emph{soliton resolution conjecture} holds for a subset of the energy space containing data with  arbitrarily large energies. 

%Our goal in this paper is to study the equivariant wave maps equation on $\Hp^2\times\R$  without any assumption on the size of the initial data. In the first part of the paper we prove that formation of singularities  for wave maps to a surface of revolution $N$ leads to bubbling off of a harmonic map from $\R^2$ to $N $. In the case where $N$ is negatively curved this implies global regularity for arbitrary finite energy data, but does not give any information about the scattering of solutions. In fact, when $N=\Hp^2$ the equation is known to admit a one parameter family of stationary  solutions $P_\la$ parametrized by their energy which varies between zero and infinity. In the second part of the paper we prove that for small values of $\la,$ corresponding to $P_\la$ with small energy, all wave maps $u(t,r)$ of arbitrarily large energy and satisfying $u(0,\infty)=P_\la(\infty)$ scatter to $P_\la$ in the energy space. This is precisely the resolution to solitons for this equation, for a certain range of the parameter $\la.$ 
\end{abstract}

\thanks{Support of the National Science Foundation, DMS-1302782 and NSF 1045119 for the first and third authors, respectively, is gratefully acknowledged. The second author is a Miller Research Fellow, and acknowledges support from the Miller Institute.}
\maketitle
%%%%%%%%%%
%%%%%%%%%%
\section{Introduction}
%%%%%%%%%%
%%%%%%%%%%%

In this article we continue the study initiated in~\cite{LOS1, LOS4} of equivariant wave maps  $U: \R \times  \Hp^2 \to \NN$, where $\Hp^2$ is the $2$-dimensional  hyperbolic space and the  target manifold   $\NN$ is a surface of revolution. 
%A map $U:  \R \times \Hp^2 \to \NN$ is called $1$-equivariant, or co-rotational, if  in polar coordinates  $(\psi(t, r) , \om)$ on the target $\NN$ with metric $\psi$  

Let $\h$ denote the metric  on hyperbolic space $\Hp^2$ and let $\eta = \textrm{diag}(-1, \h)$ denote the Lorentzian metric on $\R \times \Hp^2$.  Under the usual equivariant reduction and with polar coordinates $(t, r, \om)$ on $\R \times \Hp^2$,  a wave map takes the form $U(t, r, \om) = (\psi(t, r), \om)$ where $(\psi, \om)$ are polar coordinates on the target surface $\NN$.  The wave map system
\EQ{ \label{Ueq} 
U^{a}_{tt} - \De_{\Hp^2} U^{a}  =  \eta^{\al \be} \Ga^{a}_{bc}(U) \partial_{\al} U^b \partial_\be U^c
}   % wave map $U: I \times \Hp^2 \to \NN$ takes the form $U(t, r, \om) = ( \psi(t, r), \om)$ 
%Let $(\psi, \om)$  be polar coordinates on the target surface $\NN$ with metric $ds^2 =  d \psi^2 + g^2(\psi) d \om^2$. A $1$-equivaraint (or corotational) wave map $U: \R  \times \Hp^2 \to \NN$ can  
 reduces to a semilinear  equation for $\psi$  given by 
\EQ{\label{eq:wm} 
&\psi_{tt}  -  \psi_{rr} -  \coth r \psi_r + \frac{g(\psi)g'(\psi)}{ \sinh^2 r} = 0  \\
 %&\vec \psi(0) = (\psi_0, \psi_1)  \\
}
where the function $g$ determines the metric on $\NN$,  $ds^2 =  d \psi^2 + g^2(\psi) d \om^2$. 

The papers~\cite{LOS1, LOS4} addressed the existence and stability properties of the many finite energy stationary solutions of~\eqref{eq:wm} in the model cases  $\NN= \Sp^2$  and $\NN=\Hp^2$. When the target $\NN= \Hp^2$ it was shown in~\cite{LOS1} that there exists a $1$-parameter family of \emph{asymptotically stable} finite energy harmonic maps $P_\la$ for $\la \in [0, 1)$ with energies ranging from $0$ to $\infty$. % -- the parameter $\la \in [0, 1)$ measuring the endpoint of $P_\la$, i.e, $\lim_{r \to \infty} P_\la(r)  = 2 \arctanh(\la)$.  
 Here we focus on understanding the asymptotic dynamics of solutions to~\eqref{eq:wm} with large energy when the target is $\Hp^2$.  The space of finite energy initial data for~\eqref{eq:wm} naturally splits into disjoint classes, $\E_\la$, that are fixed by the evolution and  with a unique harmonic map $P_\la$ minimizing the energy in each class. We prove the \emph{soliton resolution conjecture} for this model for a range of $\la \in [0, 1)$.  Namely we show for a (nonperturbative) range of $\la$ that all initial data in $\E_\la$ lead to a globally regular solution that  asymptotically decouples into the unique harmonic map $P_\la$ in its energy class plus free radiation.   See Theorem~\ref{thm:main} below for a precise statement.  

A starting point in the analysis is the following simple observation: singularity formation for~\eqref{eq:wm} is a local phenomenon and thus the global geometry of the domain $\Hp^2$ does not play a role in determining the blow-up dynamics. In particular, a solution blows up in finite time by concentrating energy at the tip of a light cone centered at $r =0$.  Such a concentrating solution approximately solves the corresponding  scale invariant wave map equation on Euclidean space near the origin. In fact, it is easy to deduce the analogue of  Struwe's famous bubbling result in this context: if a wave map $U:[0, T) \times \Hp^2 \to \NN$ as above blows up at time $T< \infty$,  one can produce a sequence of maps $U_n: [0, 1) \times \R^2 \to \NN$ (obtained by translating and rescaling $U$) so that $U_n$ converges strongly in $H^1_{\loc}( [0, 1) \times \R^2 ; \NN)$ to a nontrivial finite energy  Euclidean harmonic map $Q: \R^2 \to \NN$; see Section~\ref{s:Struwe} for a sketch of this argument.   This means that if there are no finite energy nontrivial Euclidean harmonic maps $Q: \R^2 \to \NN$, then all hyperbolic wave maps $U: I \times \Hp^2 \to \NN$ are defined globally in time, i.e., $I = \R$. 
This global regularity result holds for $\NN = \Hp^2$, and thus the remaining question for this target is to describe the asymptotic behavior of a solution with arbitrary finite energy. 

%The presence of the finite energy harmonic maps means that in contrast to the corresponding problem on Euclidean space there are solutions which do not scatter to free waves as $t \to \infty$. Indeed, 

To  formulate our results in this direction we begin with a more detailed description of the model. First, as we noted above, let  $(r, \om)$ be polar coordinates on the domain $\Hp^2$ viewed as a hyperboloid in $2+1$-dimensional Minkowski space, i.e., we define a coordinate map $F$ by 
\begin{align*}
%[0, \infty) \times \Sp^1 \ni 
F:  [0, \infty) \times \Sp^1 \ni (r, \om) \mapsto (\sinh r \sin \om, \sinh r\cos \om, \cosh r) \in (\R^{2+1}, \m).
\end{align*}
where $\m$ is the Minkowski metric on $\R^{2+1}$. The  hyperbolic metric $\h$ on $\Hp^2$ in these coordinates is the pullback of the Minkowski metric by  $F$, that is  $\h = F^*\mathbf{m}$ and $ \h =  \textrm{diag}(1, \sinh^2 r)$. 
The volume element is $\sqrt{|\h(r, \om)|} \, dr d \omg = \sinh r \, dr d \omg$. For a function $f : \Hp^2 \to \R$, we have 
\begin{align*}
\int_{\Hp^2} f(x) \,  \dvol_{\h} = \int_0^{2\pi} \int_0^\I f \circ F(r ,\om) \sinh r \, dr \, d  \om.
\end{align*}
For radial functions, $f: \Hp^2 \to \R$ we abuse notation and write $f(x) = f(r)$ and we will omit the multiple of $2\pi$ obtained in the integration above.

We endow the  target surface $\NN = \Hp^2$  with polar coordinates $(\psi, \om)$.  The function~$g(\psi)$ in~\eqref{eq:wm} is  $g(\psi) = \sinh \psi$ and thus the Cauchy problem becomes 
\EQ{\label{wm} 
&\psi_{tt}  -  \psi_{rr} -  \coth r \psi_r + \frac{\sinh 2 \psi}{ 2\sinh^2 r} = 0  \\
 &\vec \psi(0) = (\psi_0, \psi_1)  \\
}
 We will often use the notation $\vec \psi(t)  := ( \psi(t), \psi_t(t))$. The conserved energy reads 
 \EQ{ \label{en} 
 \E(  \vec \psi(t))  :=  \frac{1}{2} \int_0^\infty\left(  \psi_t^2 + \psi_r^2 + \frac{ \sinh^2 \psi}{\sinh^2 r}  \right) \,  \sinh r \,  \, dr = \textrm{constant}. 
 }
 From~\eqref{en} it is clear that any finite energy data must satisfy $\psi_0(0)  = 0$. Moreover the limit $\lim_{r \to \infty} \psi_0(r)$ exists and can take any finite value -- this latter point is in stark contrast to the corresponding problem on Euclidean space where $\psi_0(r)$ must vanish at $r  = \infty$. The endpoint $ \lim_{r \to \infty} \psi_0(r)$ divides the space of finite energy data into disjoint classes, which we parameterize by $\la  \in [0, 1)$ as follows
 \EQ{
 \E_{\la}:= \{ (\psi_0, \psi_1) \mid \E( \psi_0, \psi_1)< \infty \mand  \lim_{r \to \infty} \psi_0(r) = 2 \arctanh{\la} \}.
   }
In~\cite{LOS1} we showed that for each $\la \in [0, 1)$ there is a unique harmonic map $P_{\la}(r)$ given by 
\EQ{
P_\la(r):= 2 \arctanh( \la \tanh(r/2))
}
Moreover, $(P_\la, 0)$  has minimal energy in $\E_\la$ with 
\EQ{
\E(P_\la, 0) =  \frac{ 2\la^2}{ 1- \la^2}
}
Note that $\E(P_\la, 0) \to 0$ as $\la \to 0^+$ and $\E(P_\la,  0) \to \infty$ as $\la \to 1^-$. 

 Given that $(1)$ all solutions to~\eqref{wm} are globally regular -- otherwise one could find a nontrivial Euclidean harmonic map $Q: \R^2 \to \Hp^2$, of which there are none;  and $(2)$ the energy classes $\E_\la$ are fixed by the evolution, we formulate the  following conjecture about the asymptotic behavior of solutions to~\eqref{wm}. 
\begin{conj}[Soliton resolution for equivariant wave maps $\R \times \bbH^{2} \to \bbH^{2}$] \label{c:sr} 
Consider the Cauchy problem~\eqref{wm} with finite energy initial data $(\psi_{0}, \psi_{1})$. Let $$\lmb := \tanh \frac{\psi_{0}(\infty)}{2} \in [0, 1).$$ Then~\eqref{wm} is globally well-posed, and the solution scatters to $P_{\lmb}$ as $t \to \pm \infty$.
\end{conj}

\begin{rem}
The phrase ``$\vec \psi(t)$ scatters to $P_\la$ as $t \to \pm \infty$" is defined as  follows: given a solution $\vec \psi(t)$ to~\eqref{wm} with initial data $\vec \psi(0) \in \E_\la$ we say that $\vec \psi(t)$ scatters to $P_\la$ as $t \to \pm \infty$ if there exists a solution $\fy_L^\pm$ to the linear wave equation 
\EQ{
\fy_{tt} - \fy_{rr} - \coth r \fy_r + \frac{1}{\sinh^2 r} \fy  = 0
}
so that 
\EQ{ \label{eq:2d-scat}
\|  \vec \psi(t) - (P_\la, 0) - \vec \fy^{\pm}_L(t) \|_{\HH_0} \to 0 \mas t \to \pm \infty
}
where the energy norm $\| (\cdot, \cdot)\|_{\HH_0}$ is defined as  
\ant{
\|(\phi_0,\phi_1)\|^2_{\HH_0}:=\int_0^\infty\left((\partial_r\phi_0)^2+\frac{\phi_0^2}{\sinh^2r}+\phi_1^2\right)\sinh r dr.
}

\end{rem} 

Using the celebrated concentration compactness/rigidity approach of Kenig-Merle~\cite{KM06, KM08} we are able to make partial progress on this conjecture. 
\begin{thm} \label{thm:main}
Conjecture~\ref{c:sr} holds for all initial data $(\psi_0, \psi_1) \in \E_\la$ where the endpoint $\la$ satisfies   $0 \leq \lmb \leq \Lmb$, where $\Lmb \geq 0.57716$.
\end{thm}
One of the key technical ingredients is a Bahouri-G\'erard type profile decomposition for waves on hyperbolic space established in a recent preprint~\cite{LOS2} following the work~\cite{IPS} of Ionescu, Pausader, Staffilani on the NLS.  The other main ingredient in the proof is Morawetz estimates for the linearized equation about $P_\la$, which is the main new computation in this paper.

\begin{rem}
We note that although there is a restriction on the endpoints allowed in Theorem~\ref{thm:main}, i.e, $ \la := \tanh\frac{ \psi_0( \infty)}{2} \leq \Lmb$, there is \emph{no restriction} on the energy of the data within the class $\E_\la$. This means that Theorem~\ref{thm:main} gives complete description of the asymptotic behavior of solutions with  data in admissible classes $\E_\la$, i.e., soliton resolution. 
\end{rem}
\begin{rem}
It is known even in the non-equivariant setting that all Euclidean wave maps $\R^{1+2} \to \Hp^2$ are globally regular and scatter~\cite{KS, ST1, ST2, Tao37}.  The key difference in the current setting is that the introduction of hyperbolic geometry on the domain allows for the presence of the nontrivial stationary solutions $P_\la$ and the more complicated dynamical picture outlined in Conjecture~\ref{c:sr}. 

\end{rem}

\subsection{Outline of the paper}

In Section~\ref{s:Struwe} we give a brief account of the argument needed to deduce the  version of Struwe's bubbling result outlined above. The point here is that singularity formation leads to the bubbling of a harmonic map from \emph{Euclidean space} $\R^2$ into $\NN$. For targets such as $\NN= \Hp^2$ where there are no such Euclidean harmonic maps all  solutions must be globally regular. Together with the existence of the family of asymptotically stable stationary solutions $P_\la$, this helps motivate the formulation of Conjecture~\ref{c:sr}. %since there are no obstructions for large energy perturbations of the harmonic maps to asymptotically relax to a $P_\la$ plus free radiation. 

In Section~\ref{sec: scattering} we prove Theorem~\ref{thm:main}. The proof follows the concentration compactness/rigidity method introduced by Kenig and Merle~\cite{KM06, KM08}. The key ingredients for the concentration compactness argument are Bahouri-G\'erard type profile decompositions for waves on hyperbolic space established by the authors in~\cite{LOS2}.  We begin Section~\ref{sec: scattering} by outlining the key elements of the profile decompositions along with the reduction to a critical element, i.e., if Theorem~\ref{thm:main} fails in an energy class $\E_\la$ one can find a minimal solution in $\E_\la$ which does not scatter to $P_\la$, called the a critical element.  For the rigidity argument we prove Morewetz-type estimates that rule out the possibility of such critical elements. It is here where we must restrict to the range  $0 \le \la \le \Lmb$ to maintain control over terms with indeterminate signs; see Remark~\ref{rem:Lmb}. %We note that one could potentially push the simple multiplier argument used to prove the Morawetz estimates to cover a slightly larger range of $\lambda$, however it seems that a different approach 

\section{Struwe bubbling and global regularity}\label{s:Struwe}
%%%%%%%%%%%%
%%%%%%%%%%%%%%%
In this section we  observe that any equivariant wave map  $U: [0, T_+)   \times \Hp^2 \to \NN$ that blows up in finite time leads to the bubbling of a nontrivial Euclidean harmonic map $Q: \R^2 \to \NN$ in the sense of Struwe~\cite{Struwe}. This requires some mild assumptions on the target manifold. In this section we let  $\NN$  be a surface of revolution with metric $ds^2 = d \psi^2 + g(\psi) d \om^2$, where $g$ is an smooth odd function with $g(0) = 0$ and $g'(0) =1$. If $\NN$ is  compact, then we assume that $g$ has a first zero $ \rho_0>0$ and that $g$ is periodic with period $2 \rho_0$. If $\NN$ is non-compact we assume that  $g(\rho) >0$ for $\rho>0$ and that 
\EQ{
\int_0^\infty \abs{g(\rho)} d \rho  = + \infty
}
One can  keep in mind the model compact target $\NN = \Sp^2$ where $g(\rho) = \sin \rho$ and the model non-compact target $\NN =  \Hp^2$ with $g( \rho) = \sinh \rho$. 

\begin{prop}[Struwe Bubbling] {\rm \cite[Theorem $2.1$]{Struwe} } \label{p:struwe} Let $U: [0, T_+) \times \Hp^2 \to \NN$ be a finite energy equivariant  wave map that  blows up at time $T_+ < \infty$. Then there exists a sequence of times $t_n \to T_+$ and a sequence of positive numbers $ \mu_n = o(T_+- t_n)$ so that that rescaled sequence of maps 
\EQ{ \label{Un} 
U_n(t, r, \om) := U( t_n + \mu_n t, \mu_n r, \om ) \in H^1_{\loc}( (-1, 1) \times \R^2; \NN)
}
converges locally in $H^1_{\loc}((-1, 1) \times \R^2; \NN)$ to a nontrivial, finite energy harmonic map $Q: \R^2 \to  \NN$.

\end{prop} 
\begin{rem} \label{r:r} 
 To properly interpret the sequence of functions in~\eqref{Un}  note that we are using the following slight abuse of notation:  given a radial function $ V \in L^2_{\textrm{rad}}( \Hp^2 ; \R)$, $V = V(r)$ one obtains a radial function $V \in L^2_{\textrm{rad}}(  \R^2; \R)$ by simply viewing the hyperbolic radial variable $r$ as a Euclidean distance to the origin. One then has 
\EQ{
\| V\|_{L^2(\R^2; \R)}^2  = 2\pi \int_0^\infty \abs{V(r)}^2 r \,dr \le 2\pi \int_0^\infty \abs{V(r)}^2 \, \sinh r \, dr =  \| V\|_{L^2(\Hp^2; \R)}^2
} 
\end{rem}

\begin{rem} \label{r:reg}
An immediate consequence of  Proposition~\ref{p:struwe} is that all finite energy equivariant wave maps $U: [0, T_+) \times \Hp^2 \to \Hp^2$ are globally regular, i.e., $T_+ = + \infty$, since the negative curvature of the target precludes the existence of nontrivial  finite energy harmonic maps $Q: \R^2 \to \Hp^2$, see for example~\cite[Corollary $2.2$]{Struwe}. This fact, together with the existence of the asymptotically stable harmonic maps $P_\la$ provides motivation for Conjecture~\ref{c:sr}. 

The case $\NN = \Sp^2$ is more complicated since $Q_{\euc}(r, \om)  = (2 \arctan r, \om)$ is a nontrivial finite energy harmonic map from $\R^2 \to \Sp^2$ and thus finite time blow-up is not prevented by Proposition~\ref{p:struwe}. Indeed,  the explicit blow-up constructions~\cite{KST, RR, RS} for wave maps from $\R\times \R^2 \to \Sp^2$  likely can be extended to the hyperbolic setting.  In fact, the third author~\cite{Shah1} has carried out the blow-up construction from~\cite{KST} in the setting of wave maps $\R \times \Sp^2 \to \Sp^2$ with the explicit blow-up profile given by $Q_\euc$ as above and we expect that a similar argument should hold on the hyperbolic background here. For more on the case of wave maps $\R \times \Hp^2 \to \Sp^2$  see \cite{LOS1, LOS4}. 
\end{rem}

The proof of Proposition~\ref{p:struwe} follows essentially the exact same argument as~\cite[Proof of Theorem~$2.1$]{Struwe} so we only provide a very brief sketch referring to~\cite{Struwe} for details. One additional notation we require is the local energy 
\EQ{
\E_a^b(\vec  \psi(t)) := \frac{1}{2}\int_a^b \psi_t^2(t) + \psi_r^2(t) + \frac{g^2(\psi(t))}{\sinh^2 r} \, \sinh r \, dr
}
\begin{proof}[Sketch of the proof of Proposition~\ref{p:struwe}]
Let $ U(t, r, \om) := (\psi(t,r), \om)$ be an equivariant wave map as in  Proposition~\ref{p:struwe} blowing up at time $T_+< \infty$. By translating in time we can assume that $T_+ =0$ and that our initial data is given at time $T_0 = -1$. Equivariance and energy criticality together imply that blow up must occur by a concentration of energy at the tip of the backwards light cone 
\EQ{
\CC_0:= \{ (t, r) \mid  -1 \le t \le 0, \, \, 0 \le r \le  \abs{t} \} \subset   \R \times \Hp^2
}
meaning that there exists an $\e_0>0$ so that 
\EQ{ \label{limE} 
\liminf_{ t \to 0^-} \E_0^{\abs{t}}( \vec \psi(t))  \ge \e_0
}
Next, one can exploit the local energy conservation and positivity of the flux to prove that no energy can concentrate in the self-similar region of the cone, i.e., for every $0< \mu <1$ we have 
\EQ{ \label{ss} 
 \E_{\mu \abs{t}}^{\abs{t}}( \vec \psi(t)) \to 0 \mas t \to 0^-
}
and it follows from the above that 
\EQ{\label{ta} 
\frac{1}{\abs{t}} \int_t^0 \int_0^{\abs{s}} \psi_t^2(s) \,  \sinh r \, dr \, ds \to 0  \mas t \to 0^{-}
}
The proofs of~\eqref{ss} and~\eqref{ta} are nearly  identical to the classical arguments of Christodoulou, Tahvildar-Zadeh~\cite{CTZduke} and Shatah, Tahvildar-Zadeh~\cite{STZ92}.  For the precise details we  refer the reader to \cite[Lemma 2.2]{STZ92} or \cite[Lemma 8.2]{SSbook} with the only difference here being that for the purpose of estimates one can interchange $\sinh r  \simeq  r$  and $\cosh r \simeq 1$ uniformly in the region $r \le \abs{t}$ for $\abs{t}$ small. In fact, using the notation from Remark~\ref{r:r} one can now replace $\sinh r$ with $r$ in~\eqref{limE},~\eqref{ss},~\eqref{ta} to obtain Euclidean space versions of the above  for the hyperbolic wave map $U(t, r, \om)  = (\psi(t, r), \om)$. In particular, one has  
\ant{
  &\liminf_{ t \to 0^-}  \frac{1}{2} \int_0^{\abs{t}} \left(\psi_t^2(t) + \psi_r^2(t) + \frac{g^2(\psi(t))}{r^2}\right) \, r \, dr   \ge c_0\e_0>0,\\ 
  &  \frac{1}{\abs{t}} \int_t^0 \int_0^{\abs{s}} \psi_t^2(s) \,  r \, dr \, ds \to 0  \mas t \to 0^{-}
  }
We are now in precisely the same starting point as~\cite[proof of Theorem~$2.1$]{Struwe} apart from the fact that $\vec \psi(t)$ solves~\eqref{wm} rather than its Euclidean counterpart. One can now follow the exact same argument as~\cite[proof of Lemma $3.3$]{Struwe} to find times $t_n \to 0^-$ and scales $\mu_n = o(\abs{t_n})$ so that one has 
\begin{align} 
& \frac{1}{2} \int_0^{6 \mu_{n}} \left(\psi_t^2(t_{n}) + \psi_r^2(t_{n}) + \frac{g^2(\psi(t_{n}))}{r^2}\right) \, r \, dr \geq c_{0} \e_{0} > 0 \quad \hbox{ for every } n, \label{eq:mu-nonzero} \\
& \frac{1}{\mu_n} \int_{t_n - \mu_n }^{t_n +  \mu_n} \int_0^{\abs{t}}\abs{ \partial_t \psi(t, r)}^2 \, r \, dr \, dt  \to 0 \mas  n \to \infty \label{eq:mu-T-decay} 
\end{align}

%\EQ{
% \frac{1}{\mu_n} \int_{t_n - \mu_n }^{t_n +  \mu_n} \int_0^{\abs{t}}\abs{ \partial_t \psi(t, r)}^2 \, r \, dr \, dt  \to 0 \mas  n \to \infty
% }
Then, defining 
\EQ{
U_{n}(t, r, \om) :=(\psi_n(t, r), \om) = (\psi(t_n + \mu_n t, \mu_n r), \om) = U(t_n + \mu_n t, \mu_n r,  \om) 
}
one can deduce by changing variables above that 
\EQ{
\int_{-1}^1  \int_{B(0, \abs{t_n}/ \mu_n)} \abs{\partial_t U_n(t, r, \om)}^2 r \, dr \, d\om \, dt \to 0 \mas n \to \infty
}
Following a nearly identical argument to~\cite[proof of Theorem~$2.1$]{Struwe} one can  extract from the sequence $U_n$ a strong limit $U_{\infty}$ in  $H^1_{\loc}((-1, 1) \times \R^2; \NN)$. Again, the only change to the argument presented in~\cite{Struwe} is that the $U_n$ here only approximately solves the Euclidean wave map equation on fixed compact sets in $(-1, 1) \times \R^2$, i.e., $\psi_n$ solves 
\begin{multline*}
(\partial_t^2-\lap_\euc)\psi_n+\frac{g(\psi_n)g'(\psi_n)}{r^2}=\mu_n^2\left(\coth(\mu_n r)-\frac{1}{\mu_n r}\right)\psi_r(t_n+\mu_nt,\mu_n r)\\
									     +\mu_n^2\left(\frac{1}{(\mu_n r)^2}-\frac{1}{\sinh^2(\mu_n r)}\right)g(\psi_n)g'(\psi_n).
\end{multline*}
We claim that $U_{\infty}$ is a smooth non-constant finite energy harmonic map on Euclidean space.
First, passing to the distributional limit above we see that the limiting map $U_\infty: \R^2 -  \{0\} \to \NN$ is a weak harmonic map on Euclidean space away from $r =0$. By H\'elein's theorem~\cite{Hel} and the removable singularity theorem~\cite{SU}, $U_\infty$ extends to an entire harmonic map from $\R^2 \to \NN$ and is non-constant by~\eqref{eq:mu-nonzero} and the strong $H^{1}_{\loc}((-1, 1) \times \R^{2}; \NN)$ convergence. 
\end{proof}

\section{Scattering for large energy data}\label{sec: scattering}
This section is devoted to the proof of Theorem~\ref{thm:main}. We follow the Kenig-Merle concentration compactness/rigidity method \cite{KM06, KM08}, adapted to the setting of semilinear wave equation with potential in an earlier work \cite{LOS2} of the authors.

\subsection{Reduction to a problem on $\bbH^{4}$}
We begin by reducing the problem to a semilinear wave equation with a potential on $\bbH^{4}$.

Given a finite energy wave map $\psi$, we choose $\la$ such that $ \psi(0,\infty)=P_\la(\infty)$ and let $\varphi:=\psi-P_\la.$ Then $\varphi$ satisfies
\EQ{\label{perturbed eq}
\varphi_{tt}-\varphi_{rr}-\coth r \varphi_r = F(r, \varphi),
}
where
\begin{equation*}
	F(r, \varphi) :=
	\frac{\sinh(2 P_{\lmb})(1 - \cosh(2 \varphi)) - \cosh(2 P_{\lmb}) \sinh(2 \varphi)}{2 \sinh^{2} r}.
%	\frac{\sinh(2P\la)(1-\cosh(2\varphi))+\cosh(2P_\la)(2\varphi-\sinh(2\varphi))}{2\sinh^2r}. 
\end{equation*}
%\EQ{\label{perturbed eq}
%\varphi_{tt}-\varphi_{rr}-\coth r \varphi_r+\frac{\cosh(2P_\la)}{\sinh^2r}\varphi=N(r,\varphi),
%}
%where 
%
%\ant{
%N(r,\varphi):=\frac{\sinh(2P\la)(1-\cosh(2\varphi))+\cosh(2P_\la)(2\varphi-\sinh(2\varphi))}{2\sinh^2r}.
%}
%The associated linear equation is
%\EQ{\label{lin perturbed eq}
%\left(\partial^2_{t}-\partial^2_{r}-\coth r\partial_r+\frac{\cosh(2P_\la)}{\sinh^2r}\right)\varphi_L=0.
%}
We can transform \eqref{perturbed eq} into a wave equation on $\R\times\Hp^4$ by making the change of variables $u:=\sinh^{-1}r\varphi.$ Indeed, $u$ satisfies

\ali{\label{4d eq}
&u_{tt}-u_{rr}-3\coth r u_r-2u+U_\la(r) u=\NN(r, u)\\
&\vec{u}(0):=(u(0),u_t(0))=(u_0, u_1),
}
where 
\ant{
&U_\la(r)=\frac{\cosh 2P_\la(r)-1}{\sinh^2r},\\
& \NN(r,u):=-\frac{\sinh 2P_\la}{\sinh^3r}\sinh^2(\sinh r \,u)+\frac{\cosh 2P_\la(2\sinh r \,u-\sinh(2\sinh r \,u))}{2\sinh^3r}.
}
We have moved the linear term $U_{\lmb}(r) u$ to the LHS, so as to reveal the underlying linear equation
\ali{\label{4d lin eq}
&(\partial_t^2-\lap_{\Hp^4}-2+U_\la)u_L=0.
}
We denote by $S(t) = (S_{0}(t), S_{1}(t))$ the propagator for \eqref{4d lin eq}, i.e., the function $\vec{u}_{L}(t) = (u_{L}, \rd_{t} u_{L})(t) = S(t) \vec{v}$ solves \eqref{4d lin eq} with $\vec{u}_{L}(0) = \vec{v}$. We define the \emph{linear conserved energy} $E(\vec{u})$ for \eqref{4d lin eq} as
\begin{equation*}
	E(\vec{u}) := \int_{\bbH^{4}} \abs{\nb u_{0}}^{2} + \abs{u_{1}}^{2} - 2 \abs{u_{0}}^{2} + U_{\lmb} \abs{u_{0}}^{2} \, \ud x.
\end{equation*}
Note that $E(\vec{u}_{L}(t))$ is independent of $t \in \bbR$ for a solution $\vec{u}_{L}$ to \eqref{4d lin eq}. Moreover, $E(\vec{u})$ is equivalent to $\nrm{\vec{u}}_{H^{1} \times L^{2}(\bbH^{4})}^{2}$ for every $\vec{u} \in H^{1} \times L^{2}(\bbH^{4})$.

The following lemma from \cite{LOS1} relates the $\calH_{0}$ norm of $\vec\varphi$, which enters in the scattering statement \eqref{eq:2d-scat}, with the $H^{1} \times L^{2}(\bbH^{4})$ norm of $\vec{u}$.
%implies that in order to prove Theorem~\ref{thm:main}, it suffices to prove scattering for \eqref{4d eq} with initial data $(u_0,u_1)\in H^1\times L^2(\Hp^4)$ of arbitrary size.

%%%%%%%%%
\begin{lem}[{Lemma 2.4 in \cite{LOS1}}] \label{2d to 4d} Let $(\varphi_0, \varphi_1) \in \HH_0(\Hp^2)$ with $\varphi_0(0)=0$, $\varphi_0( \infty) = 0$. Then if we define $(u_0, u_1)$ by $ (\varphi_{0}(r), \varphi_{1}(r)) = ( \sinh r \, u_0(r), \sinh r \, u_1(r))$,  we have
  \EQ{\label{H=H}
  \| (\varphi_{0}, \varphi_{1}) \|_{\HH_0}^2 \le  \|(u_0, u_1) \|_{H^1 \times L^2(\Hp^4)}^2 \le 9 \|(\varphi_{0}, \varphi_{1}) \|_{\HH_0}^2.
}
\end{lem}

%\Red{Analogue of scattering. Note that we need to also relate the potential with $\cosh (P_{\lmb})$ and the one without.}

\subsection{Concentration compactness and extraction of a critical element} \label{sec:BG}
The first step of the Kenig-Merle strategy is to establish the existence of a \emph{critical element} (i.e., a minimal non-scattering solution), under the assumption that the conclusion of Theorem~\ref{thm:main} fails.  The precise statement is as follows.
\begin{prop} \label{prop:crit-elt}
Let $0 \leq \lmb < 1$. Suppose that scattering to $P_{\lmb}$ fails for some finite energy initial data for \eqref{wm} in $\calE_{\lmb}$. Then there exists a \emph{nontrivial} solution $u_{\ast}$ to \eqref{4d eq} on $[0, \infty) \times \bbH^{2}$ whose forward trajectory $K_{+} := \set{\vec{u}_{\ast}(t) : t \in [0, \infty)}$ is pre-compact in $H^{1} \times L^{2}(\bbH^{4})$.
\end{prop}

Our proof of Proposition~\ref{prop:crit-elt} relies on the concentration compactness method \cite{KM06, KM08}. Execution of this strategy in this setting requires several ingredients, which were mostly established in the previous work \cite{LOS1, LOS2} of the authors.

We begin with perturbative theory of the nonlinear equation \eqref{4d eq}. For a time interval $I \subset \bbR$, we define the \emph{scattering norm} $S(I)$ by
\begin{equation} \label{eq:S-norm}
	\nrm{u}_{S(I)} :=  \nrm{u}_{L^{3}_{t}(I; L^{6}(\bbH^{4}))} 
\end{equation}
For simplicity of notation, we will often omit $\bbH^{4}$ in norms and write $L^{p} = L^{p}(\bbH^{4})$.

\begin{prop}[Local Cauchy theory] \label{prop:small-data}
Let $\vec{u}(0) = (u_{0}, u_{1}) \in H^{1} \times L^{2}$ be a radial initial data set, and denote by $\vec{u}(t) \in C_{t}(\bbR; H^{1} \times L^{2})$ the corresponding unique solution to \eqref{4d eq} given by Proposition~\ref{p:struwe} and Remark~\ref{r:reg}. The solution $\vec{u}(t)$ scatters as $t \to \infty$ to a free shifted wave $\vec{u}_{L}(t)$, i.e., a solution $\vec{u}_{L}(t) \in C_{t}(\bbR; H^{1} \times L^{2})$ of
\begin{equation*}
	v_{tt} - v_{rr} - 3 \coth r \, v_{r} - 2 v = 0
\end{equation*}
if and only if
\begin{equation*}
	\nrm{u}_{S([0, \infty))} + \nrm{\vec{u}}_{L^{\infty}_{t}([0, \infty); H^{1} \times L^{2})} < \infty.
\end{equation*}
An analogous statement holds in the negative time direction $t \to -\infty$.
Moreover, there exists a constant $\dlt > 0$ so that 
\begin{equation*}
	\nrm{\vec{u}(0)}_{H^{1} \times L^{2}} < \dlt \imp
	\nrm{u}_{S(\bbR)} + \nrm{\vec{u}}_{L^{\infty}_{t} (\bbR; H^{1} \times L^{2})} \aleq \nrm{\vec{u}(0)}_{H^{1} \times L^{2}}.
\end{equation*}
Hence $\vec{u}(t)$ scatters to free shifted waves as $t \to \pm \infty$.
\end{prop}

Given a function $u$ on $I \times \bbH^{4}$ such that $(u, \rd_{t} u) \in C_{t}(I; H^{1} \times L^{2})$, define
\begin{equation} \label{eq:glei}
\glei(u) := u_{tt} - u_{rr} - 3 \coth r u_{r} - 2u + U_{\lmb}(r) u - \calN(r, u)
\end{equation}
in the sense of distributions. Define also the norm $N(I)$ by
\begin{equation} \label{eq:N-norm}
	\nrm{F}_{N(I)} := \nrm{F}_{L^{1}_{t}(I; L^{2}(\bbH^{4})) + L^{\frac{3}{2}}_{t}(I; L^{\frac{12}{7}}(\bbH^{4}))}.
\end{equation}
\begin{prop}[Perturbation lemma] \label{prop:pert}
There exist non-decreasing functions \\$\eps_{0}, C_{0} : (0, \infty) \to (0, \infty)$ such that the following holds. Let $I \subset \bbR$ be an open time interval (possibly unbounded) and $t_{0} \in I$. Consider $\vec{u}, \vec{v} \in C_{t}(I; H^{1} \times L^{2})$ satisfying the bounds
\begin{align*}
	\nrm{\vec{v}}_{L^{\infty}_{t}(I; H^{1} \times L^{2})} + \nrm{v}_{S(I)} \leq & A, \\
	\nrm{\glei(u)}_{N(I)} + \nrm{\glei(v)}_{N(I)} + \nrm{w_{L}}_{S(I)} \leq & \eps,
\end{align*}
for some $A$ and $0 < \eps \leq \eps_{0}(A)$, where $w_{L} (t) := S(t - t_{0})(\vec{u} - \vec{v})(t_{0})$. Then $\vec{u}$ obeys the bound
\begin{equation} 
	\nrm{\vec{u} - \vec{v} - \vec{w}_{L}}_{L^{\infty}_{t} (I; H^{1} \times L^{2})} + \nrm{\vec{u} - \vec{v}}_{S(I)} \leq C_{0}(A) \eps.
\end{equation}
\end{prop}

Proposition~\ref{prop:small-data} was proved in \cite[Proposition~5.3]{LOS1} as a consequence of Strichartz estimates for \eqref{4d lin eq} established in \cite[Proposition~4.2]{LOS1}, which in turn relied on the previous work \cite{AP, MTay12} in the potential-free case. Proposition~\ref{prop:pert} also follows from the Strichartz estimates, by an argument similar to \cite[Proof of Lemma~5.2]{LOS2}. We remark that the necessary algebraic properties for the nonlinearity $\calN(r, u)$ are easy to verify, since $\calN$ is analytic.

A key ingredient in the proof of Proposition~\ref{prop:crit-elt} is a \emph{linear profile decomposition} for \eqref{4d lin eq}, which was first proved in the case of the wave equation on $\bbR \times \bbR^{3}$ by Bahouri and G\'erard \cite{BG}, and the Schr\"odinger equation on $\bbR \times \bbH^{3}$ by Ionescu, Pausader and Staffilani \cite{IPS}. In \cite[Theorem~4.2]{LOS2}, we proved a linear profile decomposition for a fairly general class of wave equations on $\bbR \times \bbH^{d}$, including wave equations with spectral shifts and perturbations by a time-independent potential. Here we only state a simpler version that suffices for our use.

\begin{defn}[Linear profiles] \label{def:lin-prof}
A \emph{linear profile} is a sequence of linear waves $\vec{U}_{\Box, n, L}$, constructed out of an associated sequence of parameters $\set{t_{n}, \lmb_{n}} \subseteq \bbR \times [1, \infty)$ (time and frequency) and a \emph{limiting profile}. We allow for two types of behavior of the frequencies: either $\lmb_{n} = 1$ for all $n$, or $\lim_{n \to \infty} \lmb_{n} = \infty$. Correspondingly, we define two types of linear profiles, indicated by a subscript in place of $\Box$.

\begin{itemize}[leftmargin=*]
\item {\it Case 1: A hyperbolic profile $(\Box = \hyp)$.}
Given a radial function $\vec{U}_{\hyp} \in H^{1} \times L^{2}(\Hp^4)$ (called a \emph{hyperbolic limiting profile}) and a sequence of parameters $\set{t_{n}, \lmb_{n}}$ with $\lmb_{n} = 1$ for all $n$, we define the associated \emph{hyperbolic} (or \emph{stationary}) \emph{linear profile} $\vec{U}_{\hyp, n, L}$ by
\begin{equation*}
	\vec{U}_{\hyp, n, L} (t) = S(t - t_{n}) \vec{U}_{\hyp}.
\end{equation*}

\item {\it Case 2: A Euclidean profile $(\Box = \euc)$.}
Given a radial function $\vec{V}_{\euc} \in \dot{H}^{1} \times L^{2}(\bbR^{4})$ (called a \emph{Euclidean limiting profile}) and a sequence of parameters $\set{t_{n}, \lmb_{n}}$ with $\lim_{n \to \infty} \lmb_{n} = \infty$, we define the associated \emph{Euclidean linear profile} $\vec{U}_{\euc, n, L}$ by 
\begin{equation*}
	\vec{U}_{\euc, n, L} (t):= S(t - t_{n}) \calT_{\lmb_{n}} \vec{V}_{\euc}.
\end{equation*}
Here $\calT_{\lmb_{n}}$ is a mapping from $\dot{H}^{1} \times L^{2}(\bbR^{4})$ to $H^{1} \times L^{2}(\bbH^{4})$ defined by
\begin{equation*}
	\calT_{\lmb_{n}}(f, g) (r, \omg)
	= \big( \lmb_{n} \, (\chi_{\sqrt{\lmb_{n}}} \, e^{\lmb_{n}^{-1} \lap} f) (\lmb_{n} r, \omg), 
		\lmb_{n}^{2} \, (\chi_{\sqrt{\lmb_{n}}} \, e^{\lmb_{n}^{-1} \lap} g) (\lmb_{n} r, \omg) \big),
\end{equation*}
where $\chi_{R}(r) = \chi(r/R)$ for a fixed radial function $\chi \in C^{\infty}_{0}(\bbR^{4})$, $e^{M^{-1} \lap}$ is the Euclidean Fourier multiplier with symbol $e^{-\abs{\xi}^{2} / M}$, and we identify functions on $\bbH^{4}$ and $\bbR^{4}$ by using polar coordinates $(r, \omg)$ on both spaces as in Remark~\ref{r:r}.
\end{itemize}
\end{defn}

\begin{rem} 
Unlike the Bahouri-G\'erard profile decomposition theorem on $\bbR^{1+d}$, in our case different types of linear profiles need to be distinguished. The reason is that the action of the noncompact scaling group, which is still responsible for the lack of compactness of sequence of linear waves with bounded energy, is not an actual symmetry for \eqref{4d lin eq}. We refer to \cite{IPS, LOS2} for more discussion on this point.
\end{rem}

\begin{prop}[Linear profile decomposition] \label{prop:lin-prof}
Let $\vec{u}_{n}(0)$ be a uniformly bounded sequence of initial data in $H^{1} \times L^{2}$. After passing to a subsequence, there exist parameters $\set{t_{n, j}, \lmb_{n, j}} \subseteq \bbR \times [1, \infty)$ and the corresponding linear profiles $\set{\vec{U}_{\Box, n, L}^{j}}$ (where $n, j = 1, 2, \ldots$) such that the following holds.
\begin{enumerate}[leftmargin=*]
\item {\bf Linear profile decomposition.} \label{it:lin-prof:err}
Define the error $\vec{w}^{j}_{n} \in H^{1} \times L^{2}$ by
\begin{equation*}
	\vec{u}_{n}(0) = \sum_{1 \leq j < J} \vec{U}^{j}_{\Box, n, L}(0) + \vec{w}^{J}_{n},
\end{equation*}
and let $\vec{w}^{j}_{n}(t) = S(t) \vec{w}^{j}_{n}$. Then we have
\begin{equation*}
	\limsup_{n \to \infty} \nrm{\vec{w}^{J}_{n, L}}_{S(\bbR)} \to 0 \quad \hbox{ as } J \to \infty.
\end{equation*}

\item {\bf Weak convergence property.} \label{it:lin-prof:weak}
Let $1 \leq j < J$. If $\vec{U}_{\Box, n, L}^{j} = \vec{U}_{\hyp, n, L}^{j}$ is a hyperbolic linear profile, then as $n \to \infty$,
\begin{equation*}
	S(t_{n, j}) \vec{w}_{n}^{J} \rightharpoonup 0 \quad \hbox{ weakly in } H^{1} \times L^{2}.
\end{equation*}
If $\vec{U}_{\Box, n, L}^{j} = \vec{U}_{\euc, n, L}^{j}$ is a Euclidean linear profile, then as $n \to \infty$,
\begin{equation*}
	\bb( \lmb_{n, j}^{-1} \chi(\cdot / \lmb_{n, j}^{\frac{1}{2}}) (S_{0}(t_{n, j}) \vec{w}_{n}^{J})(\cdot / \lmb_{n, j}),
	\lmb_{n, j}^{-2} \chi(\cdot / \lmb_{n, j}^{\frac{1}{2}}) (S_{1}(t_{n, j}) \vec{w}_{n}^{J})(\cdot / \lmb_{n, j}) \bb) \rightharpoonup 0
\end{equation*}
weakly in $H^{1} \times L^{2}$.

\item {\bf Orthogonality of the parameters.} \label{it:lin-prof:orth}
For each $j \neq \ell$, the sequences $\set{t_{n, j}, \lmb_{n, j}}$ and $\set{t_{n, \ell}, \lmb_{n, \ell}}$ are \emph{asymptotically orthogonal} in the sense that either
\begin{equation*}
	\bb( \frac{\lmb_{n, j}}{\lmb_{n, \ell}} + \frac{\lmb_{n,\ell}}{\lmb_{n, j}} \to \infty\bb)
	\quad \hbox{ or } \quad
	\bb( \lmb_{n, j} = \lmb_{n, \ell} \hbox{ and }
	\frac{t_{n, j} - t_{n, \ell}}{\lmb_{n, j}} \to \infty \bb).
\end{equation*}

\item {\bf Pythagorean decomposition of the energy.}  \label{it:lin-prof:en}
For each $J \geq 1$, up to passing to a subsequence, we have
\begin{equation*}
	E(\vec{u}_{n}(0)) = \sum_{1 \leq j < J} E(\vec{U}^{j}_{\Box, n, L}(0)) + E(\vec{w}_{n}^{J}) + o_{n}(1).
\end{equation*}
\end{enumerate}
\end{prop}

Proposition~\ref{prop:lin-prof} is a special case of the general result \cite[Theorem~4.2]{LOS2}.

Combining the perturbative tools (Propositions~\ref{prop:small-data} and \ref{prop:pert}) with the linear profile decomposition (Proposition~\ref{prop:lin-prof}), we may now establish a  \emph{nonlinear profile decomposition}, which approximates a sequence of nonlinear waves $\vec{u}_{n}(t)$ by superposition of nonlinear evolutions of the linear profiles, or \emph{nonlinear profiles}. We begin with a precise definition of this notion.
\begin{defn}[Nonlinear profiles] \label{def:nonlin-prof}
Let $\set{U_{\Box, n, L}}$ be a linear profile with associated parameters $\set{t_{n}, \lmb_{n}}$ as in Definition~\ref{def:lin-prof}. Assume that either
\begin{equation*}
	\lim_{n \to \infty} \lmb_{n} t_{n} = \pm \infty, 
	\quad \hbox{ or } \quad
	\lim_{n \to \infty} \lmb_{n} t_{n} = \tilde{t} \in \bbR.
\end{equation*}

\begin{itemize}[leftmargin=*]
\item {\it Case 1: A hyperbolic profile $(\Box = \hyp)$.}
We define the \emph{hyperbolic limiting nonlinear profile} $\vec{U}_{\hyp, \nl}(t)$ to be the unique solution to \eqref{4d eq} so that
\begin{equation*}
	\nrm{\vec{U}_{\hyp, n, L}(0) - \vec{U}_{\hyp, \nl}(-t_{n})}_{H^{1} \times L^{2}} \to 0 \quad \hbox{ as } n \to \infty.
\end{equation*}
The associated \emph{hyperbolic nonlinear profile} is defined to be the sequence
\begin{equation*}
	\vec{U}_{\hyp, n, \nl}(t) := \vec{U}_{\hyp, \nl}(t - t_{n}).
\end{equation*}
\item {\it Case 2: A Euclidean profile $(\Box = \euc)$.}
We define the \emph{Euclidean limiting nonlinear profile} $\vec{V}_{\euc, \nl}(t)$ to be the unique solution to the flat space equation
\begin{equation} \label{4d eq-euc}
v_{tt}-v_{rr}-\frac{3}{r} v_r =\NN_{\euc}(r, v), \quad \hbox{ where } 
\NN_{\euc}(r,v):=\frac{2 r \,v-\sinh(2 r \,v)}{2r^{3}}
\end{equation}
so that
\begin{equation*}
	\nrm{S_{\euc}(- \lmb_{n} t_{n}) \vec{V}_{\euc}(0) - \vec{V}_{\euc, \nl}(- \lmb_{n} t_{n})}_{\dot{H}^{1} \times L^{2} (\bbR^{4})}  \to 0
	\quad \hbox{ as } n \to \infty.
\end{equation*}
The associated \emph{Euclidean nonlinear profile} is defined to be the unique solution $\vec{U}_{\euc, n, \nl}(t)$ to \eqref{4d eq} with 
\begin{equation*}
	\vec{U}_{\euc, n, \nl}(t_{n}) = \calT_{\lmb_{n}} \vec{V}_{\euc, \nl}(0).
\end{equation*}
\end{itemize}
\end{defn}

The hyperbolic and Euclidean limiting nonlinear profiles $\vec{U}_{\hyp, \nl}$ and $\vec{V}_{\euc, \nl}$ are constructed by solving the respective initial value problem when $\lim_{n \to \infty} \lmb_{n} t_{n} = \tilde{t} \in \bbR$, and the scattering problem when $\lim_{n \to \infty} \lmb_{n} t_{n} = \pm \infty$.
By Proposition~\ref{p:struwe} (see also Remark~\ref{r:reg}), the solutions $\vec{U}_{\hyp, \nl}$, $\vec{U}_{\hyp, n, \nl}$ and $\vec{U}_{\euc, n, \nl}$ to \eqref{4d eq} exist globally on $\bbR \times \bbH^{4}$. By Struwe's original bubbling result \cite{Struwe}, $\vec{V}_{\euc, \nl}$ also exists globally on $\bbR \times \bbR^{4}$, and hence the above definition makes sense. 

In fact, for the Euclidean solution $\vec{V}_{\euc, \nl}$, scattering is known as well.
\begin{thm} \label{thm:euc-scat}
Let $\vec{v}(0) \in \dot{H}^{1} \times L^{2}(\bbR^{4})$, and $\vec{v}(t)$ be the corresponding solution to \eqref{4d eq-euc}. Then $\vec{v}(t)$ exists globally on $\bbR \times \bbR^{4}$ and scatters as $t \pm \infty$. More precisely, there exists a function $C : (0, \infty) \to (0, \infty)$ such that
\begin{equation*}
	\nrm{\vec{v}}_{L^{3}_{t}(\bbR; L^{6}_{x}(\bbR^{4}))} \leq C(\nrm{\vec{v}(0)}_{\dot{H}^{1} \times L^{2}(\bbR^{4})}).
\end{equation*}
\end{thm}
In this form, Theorem~\ref{thm:euc-scat} was established in \cite{CKLS1}. Note that $L^{3}_{t} (I; L^{6}(\bbR^{4}))$ is the Euclidean analogue of our $S(I)$ norm.

The Euclidean equation \eqref{4d eq-euc} arises as the vanishing scale limit of the hyperbolic space equation \eqref{4d eq}. By making this link more precise, the behavior of the Euclidean solution $\vec{V}_{\euc, \nl}$ may be connected to the hyperbolic space solutions $\vec{U}_{\euc, n, \nl}$ as follows.
\begin{prop}[Nonlinear Euclidean approximation] \label{prop:euc-evol}
Given an initial data set $\vec{v}(0) = (f, g) \in \dot{H}^{1} \times L^{2} (\bbR^{4})$ and a sequence $\set{\lmb_{n}} \subseteq [1, \infty)$ such that $\lmb_{n} \to \infty$, consider the following three objects:
\begin{itemize}[leftmargin=*]
\item Let $\vec{v}(t) \in C_{t}(\bbR; \dot{H}^{1} \times L^{2}(\bbR^{4}))$ be the nonlinear Euclidean evolution of this data, i.e., $\vec{v}(t)$ solves \eqref{4d eq-euc} with initial data $\vec{v}(0)$.
\item Applying $\calT_{\lmb_{n}} \vec{v}$ and rescaling $t$, define the sequence
\begin{equation*}
	\vec{v}_{n}(t) := \calT_{\lmb_{n}} \vec{v} (\lmb_{n} t).
\end{equation*}
\item Let $\vec{u}_{n}(t) \in C_{t}(\bbR; H^{1} \times L^{2})$ be the nonlinear hyperbolic evolution of the data $\vec{u}_{n}(0) = \calT_{\lmb_{n}}(f, g)$, i.e., $\vec{u}_{n}$ is the solution to \eqref{4d eq} with initial data $\vec{u}_{n}(0)$.
\end{itemize}
Then the following holds.
\begin{enumerate}[leftmargin=*]
\item \label{it:euc-evol:fin-t}
For an arbitrary $T_{0} > 0$, denote by $I_{n}$ the interval $(-T_{0} / \lmb_{n}, T_{0}/ \lmb_{n})$. Then 
\begin{equation*}
	\nrm{\vec{u}_{n} - \vec{v}_{n}}_{L^{\infty}_{t}(I_{n}; H^{1} \times L^{2})}
	+ \nrm{u_{n} - v_{n}}_{S(I_{n})} \to 0 \quad \hbox{ as } n \to \infty.
\end{equation*}
\item \label{it:euc-evol:sc}
For $n$ large enough, $\vec{u}_{n}$ scatters as $t \to \infty$. Moreover, there exists a function $C : (0, \infty) \to (0, \infty)$ such that
\begin{equation*}
	\limsup_{n \to \infty} \nrm{u_{n}}_{S([0, \infty))} \leq C(\nrm{(f, g)}_{\dot{H}^{1} \times L^{2}(\bbR^{4})}).
\end{equation*}

\item \label{it:euc-evol:infin-t}
Denote by $\vec{v}_{\infty}$ the free scattering data for the nonlinear Euclidean evolution $\vec{v}(t)$, i.e.,
\begin{equation*}
	\nrm{\vec{v}(t) - S_{\euc}(t) \vec{v}_{\infty}}_{\dot{H}^{1} \times L^{2}(\bbR^{4})} \to 0 \quad \hbox{ as } t \to \infty.
\end{equation*}
Then we have
\begin{equation*}
	\limsup_{n \to \infty} \nrm{\vec{u}_{n}(t) - S(t) \calT_{\lmb_{n}} \vec{v}_{\infty}}_{L^{\infty}_{t}([T/\lmb_{n}, \infty); H^{1} \times L^{2})} \to 0 \quad \hbox{ as } T \to \infty.
\end{equation*}
An analogous statement holds in the negative time direction.
\end{enumerate}
\end{prop}
Proposition~\ref{prop:euc-evol} may be proved by an argument similar to \cite[Proof of Proposition~6.4]{LOS2}, whose key ingredients are the perturbation lemma (Proposition~\ref{prop:pert}), the Euclidean scattering theorem (Theorem~\ref{thm:euc-scat}) and the linear Euclidean approximation theory developed in \cite[Section~3.3]{LOS2}. We omit the details.

For hyperbolic nonlinear profiles ($\Box = \hyp$) in Definition~\ref{def:nonlin-prof}, we obviously have
\begin{equation} \label{eq:nonlin-prof-id}
	\nrm{\vec{U}_{\Box, n, L}(-\lmb_{n} t_{n}) - \vec{U}_{\Box, n, \nl}(0)}_{H^{1} \times L^{2}} \to 0.
\end{equation}
An important consequence of Proposition~\ref{prop:euc-evol}.(\ref{it:euc-evol:fin-t}) and (\ref{it:euc-evol:infin-t}) is that \eqref{eq:nonlin-prof-id} holds for Euclidean nonlinear profiles ($\Box = \euc$) as well.
Moreover, by Proposition~\ref{prop:euc-evol}.(\ref{it:euc-evol:sc}), every Euclidean nonlinear profile $\vec{U}_{\euc, n, \nl}$ scatters and obeys
\begin{equation*}
	\limsup_{n \to \infty} \nrm{U_{\euc, n, \nl}}_{S(\bbR)} \leq C(\nrm{\vec{V}_{\euc}(0)}_{\dot{H}^{1} \times L^{2}(\bbR^{4})}) < \infty.
\end{equation*}
We are now ready to state the nonlinear profile decomposition for \eqref{4d eq}.

\begin{prop}[Nonlinear profile decomposition] \label{prop:nonlin-prof}
Consider a uniformly bounded sequence of initial data $\vec{u}_{n}(0)$ in $H^{1} \times L^{2}$. Let $\vec{U}^{j}_{\Box, n, L}$ be linear profiles for a subsequence of $\vec{u}_{n}(0)$ with associated parameters $\set{t_{n, j}, \lmb_{n, j}}$, which obey the properties stated in Proposition~\ref{prop:lin-prof}. Passing to a further subsequence if necessary, we may associate a nonlinear profile $\vec{U}_{\Box, n, \nl}^{j}$ to each $\vec{U}_{\Box, n, L}^{j}$. Denote by $\vec{u}_{n}(t)$ the unique solution to \eqref{4d eq} with $\vec{u}_{n}(0)$ as the initial data.

Let $\set{s_{n}}$ be any sequence of times so that for each $j \geq 1$ corresponding to $\Box = \hyp$, we have 
\begin{equation*}
	\limsup_{n \to \infty} \nrm{U^{j}_{\Box, \nl}}_{S([-t_{n,j}, s_{n} - t_{n,j}))} < \infty.
\end{equation*}
Then the following holds.
\begin{enumerate}[leftmargin=*]
\item {\bf Scattering norm bound for $\vec{u}_{n}(t)$.} \label{it:nonlin-scat:sc}
We have
\begin{equation*}
	\limsup_{n \to \infty} \nrm{u_{n}}_{S([0, s_{n}))} < \infty.
\end{equation*}
\item {\bf Nonlinear profile decomposition.} \label{it:nonlin-scat:err}
The error $\vec{\gmm}_{n}^{j}(t) \in C_{t}([0, s_{n}); H^{1} \times L^{2})$ defined by
\begin{equation*}
	\vec{u}_{n}(t) = \sum_{1 \leq j < J} \vec{U}_{\Box, n, \nl}^{j}(t) + \vec{w}_{n, L}^{J}(t) + \vec{\gmm}_{n}^{J}(t)
\end{equation*}
obeys
\begin{equation*}
	\limsup_{n \to \infty} \bb( \nrm{\vec{\gmm}_{n}^{J}}_{L^{\infty}_{t}([0, s_{n}); H^{1} \times L^{2})} + \nrm{\vec{\gmm}_{n}^{J}}_{S([0, s_{n}))}\bb) \to 0 \quad \hbox{ as } J \to \infty.
\end{equation*}

\item {\bf Asymptotic Pythagorean decomposition of the linear energy.} \label{it:nonlin-prof:en}
Let $\set{\tau_{n}}$ be any sequence of times such that $\tau_{n} \in [0, s_{n})$. Then up to passing to a subsequence, we have
\begin{equation*}
E(\vec{u}_{n}(\tau_{n})) = \sum_{1 \leq j < J} E(\vec{U}^{j}_{\Box, n, \nl}(\tau_{n})) + E(\vec{w}_{n, L}^{J}(\tau_{n})) + o_{n}(1).
\end{equation*}
\end{enumerate}
\end{prop}
Parts~(\ref{it:nonlin-scat:sc}) and (\ref{it:nonlin-scat:err}) of Proposition~\ref{prop:nonlin-prof} follow by applying Proposition~\ref{prop:pert} (perturbation lemma) with $\vec{u} = \vec{u}_{n}$ and $\vec{v} = \sum_{1 \leq j < J} \vec{U}_{\Box, n, \nl}^{j}$ for large $n$ and $J$, where Proposition~\ref{prop:lin-prof} and \eqref{eq:nonlin-prof-id} are used to ensure the necessary smallness condition. 
The details are similar to that of \cite[Proof of Theorem 7.2]{LOS2}.
Part~(\ref{it:nonlin-prof:en}) is a consequence of Proposition~\ref{prop:lin-prof}.(\ref{it:lin-prof:weak}) (weak convergence property) and (\ref{it:lin-prof:orth}) (orthogonality of parameters); see \cite[Proof of (8.16)]{LOS2}.

\begin{proof} [Sketch of proof of Proposition~\ref{prop:crit-elt}]
Let $0 \leq \lmb < 1$. Define
\begin{equation*}
	\calB(A) := \set{\vec{u}(0) \in H^{1} \times L^{2} : \sup_{t \in [0, \infty)} E(\vec{u}(t)) \leq A, \ \vec{u}(t) \hbox{ solves \eqref{4d eq}}}
\end{equation*}
Abusing the terminology a bit, we will refer to the norm $\sup_{t \in [0, \infty)} E(\vec{u}(t))$ as the \emph{linear energy} on $[0 ,\infty)$ of the nonlinear solution $\vec{u}(t)$ to \eqref{4d eq}. We say that $\calS \calC(A)$ holds if for all $\vec{u}(0) \in \calB(A)$ the corresponding nonlinear solution $\vec{u}(t)$ scatters forward in time, i.e., $\nrm{u}_{S([0, \infty))} < \infty$. 

Let $A_{C}$ be the infimum of all $A \in [0 ,\infty]$ for which $\calS \calC(A)$ fails. By hypothesis and Proposition~\ref{prop:small-data}, $\calS \calC(A)$ fails for some $A > 0$; hence $A_{C}$ is well-defined and finite. On the other hand, by Proposition~\ref{prop:small-data} (small data theory), $A_{C} > 0$. The key step of the proof is to prove the following compactness property for a sequence of forward-in-time nonscattering solutions whose linear energy on $[0, \infty)$ tends to~$A_{C}$.

\begin{prop}[Compactness of minimizing sequence] \label{prop:en-mini}
Let $\vec{u}_{n}(t)$ be a sequence of solutions to \eqref{4d eq} in $\calB(A)$ for some $A \in (0, \infty)$, which satisfies
\begin{equation*}
	\nrm{u}_{S([0, \infty))} = \infty \hbox{ for each } n, 
	\quad \hbox{ and } \quad
	\sup_{t \in [0, \infty) }E(\vec{u}_{n}(t)) \to A_{C}
	\hbox{ as } n \to \infty.
\end{equation*}
Then after passing to a subsequence, $\vec{u}_{n}(0)$ admits a linear profile decomposition (cf. Proposition~\ref{prop:lin-prof}) with only one nonzero profile, which is necessarily hyperbolic; we denote this profile by $\vec{U}_{\hyp, n, L}$ and the associated time parameters by $t_{n}$ (the frequency parameters $\lmb_{n}$ are all equal to $1$). Moreover, the error $\vec{w}_{n} \in H^{1} \times L^{2}$ defined by
\begin{equation*}
	\vec{u}_{n}(0) = \vec{U}_{\hyp, n, L}(0) + \vec{w}_{n}
\end{equation*}
vanishes in $H^{1} \times L^{2}$, i.e.,
\begin{equation*}
	\lim_{n \to \infty} \nrm{\vec{w}_{n}}_{H^{1} \times L^{2}} = 0.
\end{equation*}
\end{prop}
The proof of Proposition~\ref{prop:en-mini} begins with an application of Proposition~\ref{prop:lin-prof} and \ref{prop:nonlin-prof}. If there are more than one nonscattering nonlinear profiles, then $\nrm{\vec{u}_{n}}_{S([0, \infty))} < \infty$ for large $n$ (up to taking a subsequence). Hence there is exactly one nonscattering nonlinear limiting profile, and by Proposition~\ref{prop:euc-evol}, it is necessarily hyperbolic. Finally, by Proposition~\ref{prop:nonlin-prof}.(\ref{it:lin-prof:en}), the energy minimizing property of $\vec{u}_{n}$ and Proposition~\ref{prop:small-data} (small data theory), vanishing of all other profiles and $\vec{w}^{J}_{n}$ in $L^{\infty}_{t}([0, \infty); H^{1} \times L^{2})$ follows. For details, we refer to \cite[Proof of Proposition~8.2]{LOS2}. 

With Proposition~\ref{prop:en-mini} in our hand, the proof of Proposition~\ref{prop:crit-elt} may be completed in a few strokes.

 \vskip.5em
\noindent {\bf Step 1: Existence of a critical element.} 
Consider a sequence of the initial data sets for forward-in-time nonscattering solutions $\vec{u}_{n}(t)$ that minimizes the linear energy on $[0, \infty)$. We define the \emph{critical element} $\vec{u}_{\ast}(t)$ to be the nonlinear limiting profile associated to the single nonzero hyperbolic linear profile given by Proposition~\ref{prop:en-mini} applied to this sequence. By Proposition~\ref{prop:pert} and the hypotheses on $\vec{u}_{n}(t)$, for every $\tau \geq 0$ the critical element $\vec{u}_{\ast}(t)$ obeys
\begin{equation} \label{eq:crit-elt-prop}
	\nrm{u_{\ast}}_{S([\tau, \infty)} = \infty, \quad \hbox{ and } \quad
	\sup_{t \in [\tau, \infty)} E(\vec{u}_{\ast}(t)) = A_{C}. 
\end{equation}

\vskip.5em
\noindent {\bf Step 2: Compact trajectory property.}
To complete the proof of Proposition~\ref{prop:crit-elt}, it only remains to verify that the forward trajectory $K_{+}$ of $\vec{u}_{\ast}$ is pre-compact in $H^{1} \times L^{2}$, or equivalently, that any sequence $\set{\vec{u}_{\ast}(\tau_{n}) } \subseteq K_{+}$ contains a convergent subsequence. 

We may assume that $\tau_{n} \to \infty$, since the statement obviously
 holds otherwise by continuity of the trajectory $t \mapsto \vec{u}_{\ast}(t)$. Then, we may apply Proposition~\ref{prop:en-mini} to $\vec{u}_{\ast}(\tau_{n})$ thanks to \eqref{eq:crit-elt-prop}.  Passing to a subsequence and arguing as before, 
\begin{equation*}
	\vec{u}_{\ast}(\tau_{n}) = \vec{U}_{\hyp, L}(-t_{n}) + \vec{w}_{n},
\end{equation*}
where $\nrm{\vec{w}_{n}}_{H^{1} \times L^{2}} \to 0$. We claim that $\limsup_{n \to \infty} \abs{t_{n}} < \infty$; then passing to a subsequence so that $t_{n} \to t_{0} \in \bbR$, $\vec{u}_{\ast}(\tau_{n})$ converges to $\vec{U}_{\hyp, L}(t_{0})$, thereby establishing the desired conclusion.

Suppose that the claim fails; then on a subsequence, we have $t_{n} \to -\infty$ or $t_{n} \to \infty$.
In the former case, $\nrm{S(t) \vec{u}_{\ast}(\tau_{n})}_{S([0, \infty))} \to 0$, but this implies by Proposition~\ref{prop:small-data} that $\nrm{u_{\ast}}_{S([0, \infty))} < \infty$, which is impossible. 
In the latter case, $\nrm{U_{\hyp, L}(t-t_{n})}_{S((-\infty, 0])} \to 0$, and hence by Proposition~\ref{prop:pert}, $\nrm{u_{\ast}}_{S((-\infty, \tau_{n}])}$ is uniformly bounded in $n$. Since $\tau_{n} \to \infty$, we again deduce $\nrm{u_{\ast}}_{S([0, \infty))} < \infty$, which is a contradiction. \qedhere
\end{proof}

\subsection{Morawetz estimate} \label{s:mor} 
In this subsection, we set the stage for the second step of the Kenig-Merle approach, by establishing a Morawetz-type estimate for finite energy wave maps. 
\begin{prop} \label{prop:morawetz}
There exists $\Lmb \geq 1/2$ such that the following holds.
Let $\lmb \in [0, \Lmb]$ and let $\vec{\psi}(t) \in C_{t}([0, \infty); \calE_{\lmb})$ be an equivariant wave map with finite energy $\calE$. Denote by $\vec{u}(t) = \sinh^{-1} r \, [\vec{\psi}(t) - (P_{\lmb}(r), 0)]$ the corresponding $4d$ variable. Then
\begin{equation} \label{eq:morawetz}
	\int_{I} \int_{0}^{\infty} \abs{u}^{4} \sinh^{3} r \, \ud r \, \ud t \leq C \calE.
\end{equation}
\end{prop}

Unfortunately, our strategy of proof does not allow us to reach $\lmb = 1$, and limits the validity of Theorem~\ref{thm:main} to $0 \leq \lmb \leq \Lmb < 1$. The optimal value of $\Lmb$ arising from our proof, which was stated in Theorem~\ref{thm:main}, is $0.57716\cdots$; see Remark~\ref{rem:Lmb} for more discussion.

\begin{proof} [Proof of Proposition~\ref{prop:morawetz}]
 In order to prove Proposition~\ref{prop:morawetz}, we work with the variable $\varphi = \psi - P_{\lmb} = \sinh r \, u$ on $\bbH^{2}$. In terms of $\varphi$, \eqref{eq:morawetz} is equivalent to 
\begin{equation} \label{eq:morawetz-2d}
	\int_{I} \int_{0}^{\infty} \frac{\abs{\varphi}^{4}}{\sinh^{2} r} \sinh r \, \ud r \, \ud t \leq C \calE.
\end{equation}
 We introduce the radial multiplier\footnote{Note that this definition specifies $a$ only up to a constant, but this constant will not play any role.} $a : \bbH^{2} \to \bbR$ with
\begin{equation*}
	a_{r} = \frac{\cosh r - 1}{\sinh r} \,.
\end{equation*}
Its key properties are that $\lap_{\bbH^{2}} a =1,$ $a_{r}$ is bounded, and $a_{rr}\geq0$. In fact, we have
\begin{equation} \label{eq:est4ar}
	0 \leq a_{r} \leq 1, \quad 
	\frac{1}{2} \leq a_{r} \coth r \leq 1.
\end{equation}
Using the equation \eqref{perturbed eq} for $\varphi$, the following multiplier identity corresponding to $a_{r}$ can be derived:
\begin{align*}
\frac{\ud}{\ud t} \brk{\varphi_{t} \mid a_{r} \varphi_{r}} 
%=&	\frac{1}{2} \int_{0}^{\infty} a_{r} \rd_{r} (\varphi_{t}^{2}) \sinh r \, \ud r
%	+ \int_{0}^{\infty} a_{r} \varphi_{r} \bb( \frac{1}{\sinh r} \rd_{r} (\sinh r \varphi_{r}) - \frac{F(P_{\lmb}, \varphi)}{\sinh^{2} r} \bb) \sinh r \, \ud r \\
%=&	- \frac{1}{2} \int_{0}^{\infty} \varphi_{t}^{2} \sinh r \, \ud r
%	- \int_{0}^{\infty} a_{rr} \varphi_{r}^{2} \sinh r \, \ud r
%	- \frac{1}{2} \int_{0}^{\infty} a_{r} \rd_{r} (\varphi_{r}^{2})  \, \ud r
%	- \int_{0}^{\infty} a_{r} \varphi_{r} \frac{F(P_{\lmb}, \varphi)}{\sinh^{2} r}  \sinh r \, \ud r \\
=&	- \frac{1}{2} \int_{0}^{\infty} (\varphi_{t}^{2} - \varphi_{r}^{2}) \sinh r \, \ud r
	- \int_{0}^{\infty} a_{rr} \varphi_{r}^{2} \sinh r \, \ud r \\
&	- \int_{0}^{\infty} a_{r} \varphi_{r} \frac{F(P_{\lmb}, \varphi)}{\sinh^{2} r}  \sinh r \, \ud r.
\end{align*}
On the other hand, we also have the identity
\begin{align*}
\frac{\ud}{\ud t} \brk{\varphi_{t} \mid \frac{\varphi}{2}}
%= \frac{1}{2} \int_{0}^{\infty} \varphi_{t}^{2} \sinh r \, \ud r + \frac{1}{2} \int_{0}^{\infty} \varphi \bb( \frac{1}{\sinh r} \rd_{r} (\sinh r \varphi_{r}) - \frac{F(Q, \varphi)}{\sinh^{2} r} \bb) \sinh r \, \ud r \\
= \frac{1}{2} \int_{0}^{\infty} (\varphi_{t}^{2} - \varphi_{r}^{2}) \sinh r \, \ud r 
- \frac{1}{2} \int_{0}^{\infty} \varphi  \frac{F(P_\lambda, \varphi)}{\sinh^{2} r} \sinh r \, \ud r.
\end{align*}
Adding the two preceding identities, we obtain 
\begin{align*}
\frac{\ud}{\ud t} \brk{\varphi_{t} \mid a_{r} \varphi_{r} + \frac{\varphi}{2}}
= 	& - \int_{0}^{\infty} a_{rr} \varphi_{r}^{2} \sinh r \, \ud r
	- \frac{1}{2} \int_{0}^{\infty} \varphi  \frac{F(P_\lambda, \varphi)}{\sinh^{2} r} \sinh r \, \ud r \\
	& -  \int_{0}^{\infty} a_{r} \varphi_{r} \frac{F(P_\lambda, \varphi)}{\sinh^{2} r}  \sinh r \, \ud r.
\end{align*}
For the last integral, we furthermore plug in the formula
\begin{align*}
	\rd_{r} \varphi F(P_{\lmb}, \varphi)
=& \frac{1}{2} \bb( \frac{1}{2} \cosh(2P_{\lmb}) \rd_{r} (\cosh(2 \varphi) - 1)  + \frac{1}{2} \sinh(2P_{\lmb}) \rd_{r} (\sinh(2 \varphi) - 2\varphi) \bb),
\end{align*}
and integrate by parts in $r$. As a result, we obtain 
\begin{equation} \label{eq:mwtz:general}
\begin{aligned}
\frac{\ud}{\ud t} \brk{\varphi_{t} \mid a_{r} \varphi_{r} + \frac{\varphi}{2}}
%=& 	- \int_{0}^{\infty} a_{rr} \varphi_{r}^{2} \sinh r \, \ud r 	\\
%&	- \frac{1}{4} \int_{0}^{\infty} \varphi  \bb( C(2Q) S(2 \varphi) \pm 2 S(2Q) S(\varphi)^{2} \bb) \frac{1}{\sinh^{2} r} \sinh r \, \ud r \\
%&	-  \frac{1}{2} \int_{0}^{\infty} \bb( C(2Q) \rd_{r} S(\varphi)^{2} + S(2Q) \rd_{r} (\frac{1}{2} S(2 \varphi) - \varphi) \bb) \frac{1}{\sinh^{2} r} a_{r} \sinh r \, \ud r \\
%=& 	- \int_{0}^{\infty} a_{rr} \varphi_{r}^{2} \sinh r \, \ud r 	\\
%&	- \frac{1}{4} \int_{0}^{\infty}   \bb( C(2Q) S(2 \varphi) \varphi + S(2Q) (C(2\varphi) - 1) \varphi \bb) \frac{1}{\sinh^{2} r} \sinh r \, \ud r \\
%&	+ \frac{1}{4} \int_{0}^{\infty} \bb( \pm C(2Q) (C(2 \varphi) -1 )  + S(2Q) (S(2 \varphi) - 2\varphi) \bb) \frac{1}{\sinh^{2} r} \sinh r \, \ud r \\
%&	-  \frac{1}{2} \int_{0}^{\infty} \bb( \pm C(2Q) (C(2 \varphi) -1 )  + S(2Q) (S(2 \varphi) - 2\varphi) \bb) \frac{1}{\sinh^{2} r} \coth r a_{r} \sinh r \, \ud r \\
%&	+  \frac{1}{2} \int_{0}^{\infty} \bb( S(2Q) Q_{r} (C(2 \varphi) - 1) + C(2Q) Q_{r} (S(2 \varphi) - 2 \varphi) \bb) \frac{1}{\sinh^{2} r} a_{r} \sinh r \, \ud r 
=& 	- \int_{0}^{\infty} a_{rr} \varphi_{r}^{2} \sinh r \, \ud r - \frac{1}{4} I[\varphi]	
\end{aligned}
\end{equation}
where $I[\varphi]$ is given by
\begin{equation*} 
\begin{aligned}
I[\varphi] := &	 \int_{0}^{\infty}  \bb( \cosh(2P_{\lmb}) \sinh(2 \varphi) \varphi + \sinh(2P_{\lmb}) (\cosh(2 \varphi) -1) \varphi \bb) \frac{1}{\sinh r}  \, \ud r \\
&	- \int_{0}^{\infty} \bb( \cosh (2 P_{\lmb}) (\cosh(2 \varphi) -1) \varphi + \sinh (2P_{\lmb}) (\sinh (2 \varphi) - 2\varphi) \bb) \frac{1}{\sinh r} \, \ud r \\
&	+  \int_{0}^{\infty} \bb( \cosh(2P_{\lmb}) (\cosh(2 \varphi) -1) + \sinh(2P_{\lmb}) ( \sinh(2 \varphi) - 2 \varphi) \bb)  \frac{2\coth r a_{r} }{\sinh r} \, \ud r \\
&	- \int_{0}^{\infty} \bb( \sinh(2P_{\lmb}) \frac{\sinh P_{\lmb}}{\cosh r} (\cosh(2 \varphi) -1)  \bb) 2 \coth r a_{r} \frac{1}{\sinh r} \, \ud r \\
&	- \int_{0}^{\infty} \bb( \cosh (2P_{\lmb}) \frac{\sinh P_{\lmb}}{\cosh r} (\sinh (2 \varphi) - 2 \varphi) \bb) 2 \coth r a_{r} \frac{1}{\sinh r} \, \ud r .
\end{aligned}
\end{equation*}

Upon integration of \eqref{eq:mwtz:general} in $t$ over any time interval $I$, the contribution of the LHS can be uniformly bounded by the conserved energy for the full wave map $\psi$. The first term on the RHS has a sign, so can be thrown away. Hence the desired estimate \eqref{eq:morawetz} would follow once we prove
\begin{equation} \label{eq:mwtz:H2:mid-goal}
	I[\varphi] \geq c_{\lmb} \int_{0}^{\infty} \frac{\abs{\varphi}^{4}}{\sinh^{2} r} \sinh r \, \ud r
\end{equation}
for some $c_{\lmb} > 0$.  We will establish this inequality for $0 \leq \lmb \leq \Lmb$, where $\Lmb > 0$ is an explicit constant which can be shown to be $\Lmb \geq 0.56$. We remark that $c_{\lmb}$ is given by \eqref{eq:mwtz:c lmb}.

In order to proceed, we reduce the problem further to verifying a pointwise bound. Consider the expression
\begin{align}
&	 \sinh(2 \varphi) \varphi + \tanh (2P_{\lmb}) (\cosh(2 \varphi) -1) \varphi  \notag \\
&	- (\cosh(2 \varphi) -1) - \tanh (2P_{\lmb}) (\sinh (2 \varphi) - 2\varphi) \label{eq:mwtz:density4H2}\\
&	+  \bb( (\cosh(2 \varphi) -1) + \tanh(2P_{\lmb}) (\sinh(2 \varphi) - 2 \varphi) \bb) 2 \coth r a_{r} \notag \\
&	- \bb(  \tanh(2P_{\lmb}) \frac{\sinh P_{\lmb}}{\cosh r} (\cosh(2 \varphi) -1) +  \frac{\sinh P_{\lmb}}{\cosh r} (\sinh (2 \varphi) - 2 \varphi) \bb) 2 \coth r a_{r} , \notag
\end{align}
which is the integrand in the definition of $I[\varphi]$ divided by $\frac{\cosh (2P_{\lmb})}{\sinh r}$. Note that $\cosh (2 P_{\lmb}) \geq 1$ for all $\lmb \in [0, 1)$ and $r \in (0, \infty)$. Therefore, \eqref{eq:mwtz:H2:mid-goal} reduces to
\begin{equation} \label{eq:mwtz:H2:goal}
	\eqref{eq:mwtz:density4H2} \geq c_{\lmb} \abs{\varphi}^{4}.
\end{equation}

The proof of \eqref{eq:mwtz:H2:goal} is split into two cases: when $\varphi \geq 0$ and $\varphi < 0$.

\vskip.5em
\noindent {\bf Case 1: $\varphi \geq 0$.}
In this case, we can treat all $\lmb \in [0, 1)$. Here, it suffices to know the following trivial bound for $P_{\lmb}$:
\begin{equation*}
	0 \leq P_{\lmb}(r) = 2 \arctanh (\lmb \tanh(\frac{r}{2}) ) \leq r.
\end{equation*}

The desired claim follows from the four inequalities below:
\begin{align} 
	\sinh (2 \varphi) \varphi - \cosh (2 \varphi) + 1  \geq \frac{2}{3} \varphi^{4} , \label{eq:mwtz:H2:pos1}\\
	\tanh(2P_{\lmb}) \bb( (\cosh (2\varphi) - 1) \varphi - ( \sinh (2 \varphi) - 2 \varphi) \bb) \geq 0 , \label{eq:mwtz:H2:pos2}\\
	2 \coth r a_{r} \bb( 1  - \tanh (2 P_{\lmb}) \frac{\sinh P_{\lmb}}{\cosh r} \bb) (\cosh (2 \varphi) - 1)   \geq 0 ,\label{eq:mwtz:H2:pos3}\\
	2 \coth r a_{r} \bb( \tanh(2 P_{\lmb}) - \frac{\sinh P_{\lmb}}{\cosh r} \bb) ( \sinh (2 \varphi) - 2 \varphi ) \geq 0. \label{eq:mwtz:H2:pos4}
\end{align}

We now turn to the proofs of these inequalities.

\vskip.3em
\noindent{\it Case 1.1: Proof of \eqref{eq:mwtz:H2:pos1}.}
This inequality is an easy consequence of Taylor expansion:
\begin{align*}
	\sinh (2 \varphi) \varphi - \cosh(2 \varphi) + 1
%= & \sum_{k=0}^{\infty} \frac{1}{2 (2k+1)!} (2 \varphi)^{2k+2} - \sum_{k=1}^{\infty} \frac{1}{(2k)!} (2\varphi)^{2k} \\
= & \sum_{k=0}^{\infty} \bb( \frac{1}{2 (2k+1)!}  - \frac{1}{(2k+2)!} \bb) (2\varphi)^{2k+2} \\
\geq & \frac{2}{3} \varphi^{4} + \sum_{k=2}^{\infty} \bb( \frac{1}{(2k+1)!} \frac{k}{2k+2} \bb) (2\varphi)^{2k + 2}
\geq \frac{2}{3} \varphi^{4}.
\end{align*}

\vskip.3em
\noindent{\it Case 1.2: Proof of \eqref{eq:mwtz:H2:pos2}.}
By Taylor expansion, we have
\begin{align*}
\frac{1}{2} (\cosh 2\varphi - 1) 2\varphi - ( \sinh (2 \varphi) - 2 \varphi) 
=& \sum_{k=1}^{\infty} \bb( \frac{1}{2 (2k)!} - \frac{1}{(2k+1)!}\bb)(2 \varphi)^{2k+1} \geq 0.
\end{align*}
Then \eqref{eq:mwtz:H2:pos2} follows immediately.

\vskip.3em
\noindent{\it Case 1.3: Proof of \eqref{eq:mwtz:H2:pos3}.}
It suffices to prove
\begin{align*}
	1  - \tanh (2 P_{\lmb}) \frac{\sinh P_{\lmb}}{\cosh r} \geq 0,
\end{align*}
which is obvious since $0 \leq \tanh (2 P_{\lmb}) \leq 1$, $0 \leq \frac{\sinh P_{\lmb}}{\cosh r} \leq \tanh r \leq 1$.

\vskip.3em
\noindent{\it Case 1.4: Proof of \eqref{eq:mwtz:H2:pos4}.}
Note that $\sinh (2 \varphi) - 2\varphi \geq 0$, since $\varphi \geq 0$. Therefore, it suffices to prove
\begin{equation*}
\tanh(2 P_{\lmb}) - \frac{\sinh P_{\lmb}}{\cosh r} \geq 0
\end{equation*}
which follows from
\begin{equation*}
\tanh(2 P_{\lmb}) \geq \tanh P_{\lmb} = \frac{\sinh P_{\lmb}}{\cosh P_{\lmb}} \geq \frac{\sinh P_{\lmb}}{\cosh r}.
\end{equation*}

\vskip.5em
\noindent{\bf Case 2: $\varphi < 0$.} In this case, we can only treat $\lmb$ which are not too large. For convenience, we define $\phi := - \varphi$, so that $\phi > 0$. Then \eqref{eq:mwtz:density4H2} becomes
\begin{equation} \label{eq:mwtz:density4H2:neg} \tag{\ref{eq:mwtz:density4H2}$'$}
\begin{aligned}
&	\sinh(2 \phi) \phi - \tanh (2P_{\lmb}) (\cosh (2 \phi) - 1) \phi  \\
&	- (\cosh (2 \phi) - 1) + \tanh (2P_{\lmb}) (\sinh (2 \phi) - 2\phi) \\
&	+  \bb( (\cosh (2 \phi) - 1) - \tanh(2P_{\lmb}) (\sinh(2 \phi) - 2 \phi) \bb) 2 \coth r a_{r} \\
&	- \bb(  \tanh(2P_{\lmb}) \frac{\sinh P_{\lmb}}{\cosh r} (\cosh (2 \phi) - 1) -  \frac{\sinh P_{\lmb}}{\cosh r} (\sinh (2 \phi) - 2 \phi) \bb) 2 \coth r a_{r} 
\end{aligned}
\end{equation}
We now separate each term that does not have a factor of $\tanh (2 P_{\lmb})$ into $(1- \tanh (2 P_{\lmb})) + \tanh(2 P_{\lmb})$, and collect all terms with a factor of $\tanh (2 P_{\lmb})$. The desired claim is now an easy consequence of the following three inequalities:
\begin{align} 
	& (1-\tanh (2P_{\lmb}) ) \bb( \sinh (2 \phi) \phi - (\cosh (2 \phi) - 1) \bb) \geq c_{\lmb} \phi^{4}, \label{eq:mwtz:H2:neg1} \\
	 & 2 \coth r a_{r} (1- \tanh(2 P_{\lmb})) \bb( (\cosh (2 \phi) - 1) + \frac{\sinh P_{\lmb}}{\cosh r} ( \sinh (2 \phi) - 2\phi ) \bb) \geq 0, \label{eq:mwtz:H2:neg2} \\
	& \tanh(2 P_{\lmb}) \bb[ \bb( \sinh(2 \phi) \phi - (\cosh (2 \phi) - 1) \phi - (\cosh(2 \phi) - 1) + (\sinh(2\phi) - 2\phi) \bb) \quad \phantom{\geq} \notag \\
	& + 2 \coth r a_{r}  \bb( 1 - \frac{\sinh P_{\lmb}}{\cosh r} \bb) \bb( (\cosh (2 \phi) - 1) - (\sinh (2\phi) - 2 \phi) \bb) \bb] \geq 0. \label{eq:mwtz:H2:neg3}
\end{align}
where 
\begin{equation} \label{eq:mwtz:c lmb}
	c_{\lmb} = \frac{2}{3} (1-\tanh (4 \arctanh \lmb)).
\end{equation}

\vskip.3em
\noindent{\it Case 2.1: Proof of \eqref{eq:mwtz:H2:neg1}.}
This inequality follows from \eqref{eq:mwtz:H2:pos1}.

\vskip.3em
\noindent{\it Case 2.2: Proof of \eqref{eq:mwtz:H2:neg2}.}
This inequality is obvious for $\phi \geq 0$.

\vskip.3em
\noindent{\it Case 2.3: Proof of \eqref{eq:mwtz:H2:neg3}.}
We claim that \eqref{eq:mwtz:H2:neg3} holds if
\begin{equation} \label{eq:mwtz:H2:key}
	2 \coth r a_{r}  \bb( 1 - \frac{\sinh P_{\lmb}}{\cosh r} \bb) \geq \frac{3}{4}.
\end{equation}
Indeed, suppose that \eqref{eq:mwtz:H2:key} holds. Since $\tanh (2 P_{\lmb}) \geq 0$, it suffices to show that the expression inside the square brackets is non-negative. Using \eqref{eq:mwtz:H2:key}, this expression can be bounded from below by
\begin{align*}
	\geq & \bb( \sinh(2 \phi) \phi - (\cosh (2 \phi) - 1) \phi - (\cosh(2 \phi) - 1) + (\sinh(2\phi) - 2\phi) \bb) \\
	& + \frac{3}{4} \bb( (\cosh (2 \phi) - 1) - (\sinh (2\phi) - 2 \phi) \bb) \\
	= & - (e^{-2 \phi} - 1) \phi - \frac{1}{4} (e^{- 2 \phi} + 2 \phi - 1)
\end{align*}
Note that the last expression equals $0$ when $\phi = 0$, and its $\phi$-derivative equals $\frac{1}{2} (1 - e^{-2 \phi}) + 2 \phi e^{-2 \phi}$, which is manifestly non-negative for $\phi \geq 0$. Therefore, \eqref{eq:mwtz:H2:neg3} follows by integration.

It now remains to find $\Lmb$ such that \eqref{eq:mwtz:H2:key} holds in the range $\lmb \in [0, \Lmb]$. Observe that this inequality obviously holds for sufficiently small $\lmb \geq 0$, since $2 \coth r a_{r} \geq \frac{1}{2}$ and the term $\frac{\sinh P_{\lmb}}{\cosh r}$ converges uniformly to $0$ as $\lmb \to 0$. To simply the expression, we employ the algebraic trick\footnote{In fact, geometrically, this trick amounts to working with the Poincar\'e disk model for the domain.} of working with a new variable
\begin{equation*}
	s := \tanh (r/2).
\end{equation*}

Then we have the following formulae:
\begin{align*}
	\cosh r = \frac{1 + s^{2}}{1 - s^{2}}, \quad
	\sinh P_{\lmb} = \frac{2 \lmb s}{1-\lmb^{2} s^{2}}, \quad
	2 a_{r} \coth r = (1+s^{2})
\end{align*}
and \eqref{eq:mwtz:H2:key} becomes
\begin{align*}
	(1+s^{2}) \bb( 1 - \frac{2 \lmb s}{1 - \lmb^{2} s^{2}} \frac{1-s^{2}}{1+s^{2}} \bb) 
%=	(1+s^{2}) \bb( \frac{(1+s^{2})(1-\lmb^{2} s^{2})}{(1+s^{2})(1-\lmb^{2} s^{2})} - \frac{2 \lmb s (1-s^{2})}{(1+s^{2})(1 - \lmb^{2} s^{2})}  \bb) 
%=	 \frac{(1+s^{2})(1-\lmb^{2} s^{2}) - 2 \lmb s (1-s^{2})}{(1-\lmb^{2} s^{2})} 
	\geq \frac{3}{4}
\end{align*}
or equivalently,
\begin{equation} \label{eq:final-ineq}
%(1+s^{2})(1-\lmb^{2} s^{2}) - 2 \lmb s (1-s^{2}) 
%	\geq \frac{3}{4} (1-\lmb^{2} s^{2})
\bb( \frac{1}{4}+s^{2} \bb)(1-\lmb^{2} s^{2}) - 2 \lmb s (1-s^{2}) 
	\geq 0
\end{equation}
for $s \in (0, 1)$. Using calculus, it may be easily verified that $\rd_{\lmb}$ of the LHS of \eqref{eq:final-ineq} is non-positive for $s, \lmb \in (0, 1)$. Hence if \eqref{eq:final-ineq} holds for some $\lmb_{0} \in (0, 1)$, then it holds for all $0 \leq \lmb \leq \lmb_{0}$. Let $\Lmb$ be the supremum of such $\lmb_{0}$. It is easy to see that $\Lmb \geq 1/2$, since the LHS of \eqref{eq:final-ineq} in the case $\lmb = 1/2$ becomes
\begin{equation*}
\bb( \frac{1}{4}+s^{2} \bb)\bb( 1- \frac{ s^{2}}{4} \bb) - s (1-s^{2}) 
\geq 	\bb( \frac{1}{4} - s +s^{2} \bb)(1-s^{2}) \geq 0
\end{equation*}
for $s \in (0, 1)$. 
\end{proof}

\begin{rem} \label{rem:Lmb}
Unfortunately, our naive strategy of proving a pointwise bound does not reach $\lmb = 1$. Indeed, $(s, \lmb) = (1/2, 3/4)$ violates \eqref{eq:final-ineq}, which implies that $\Lmb < 3/4$. 
With the help of a computer algebra system, it may be checked that $\Lmb = 0.57716\cdots$.
\end{rem}

\subsection{Rigidity and conclusion of proof of Theorem~\ref{thm:main}}
With Proposition~\ref{prop:morawetz}, we are ready to perform the rigidity step of the Kenig-Merle approach, and conclude the proof of Theorem~\ref{thm:main}.
\begin{proof}[Proof of Theorem~\ref{thm:main}]
Fix $\la\leq\Lambda$ where $\Lambda$ is as in Proposition~\ref{prop:morawetz}. If the theorem fails, then by Proposition~\ref{prop:crit-elt} we can find a critical element $\vec{u}_{\ast}$ such that $\vec{u}_{\ast} \not \equiv (0, 0)$ and the forward trajectory $K_{+}$ is pre-compact in $H^{1} \times L^{2}$.  Let $\vec{\psi}_{\ast} = \sinh r \vec{u}_{\ast} + (P_{\lmb}, 0)$ be the corresponding equivariant wave map with energy $\calE$.
We will argue as in the proof of Proposition 8.4 in \cite{LOS2} to show that $\vec{u}_{\ast}(t)\equiv (0,0),$ which is a contradiction. 

All the constants in the proof may depend on the energy $\E$ of the wave map $\vec{\psi}_{\ast}$. 
Let $K_M$ be a radial kernel supported on the ball $\{r\leq\frac{1}{M}\}$ in $\Hp^4$, uniformly bounded for $M\geq1$, and satisfying $\int_{\Hp^4}K_M d\textrm{Vol}_{\Hp^4}=1$. Define the approximation of the identity $Q_M f=K_M\ast f$. To define the convolution, we use the group structure coming from the representation of the hyperbolic space as a symmetric space. See \cite{Bray} for more details. 

Fix $\epsilon>0.$ Because $K_{+}$ is pre-compact, if $M$ is sufficiently large then
\ali{\label{small Strich 1}
\sup_{t \in [0, \infty)}\|(1-Q_M)u_{\ast}(t)\|_{L^4(\Hp^4)}\leq\frac{\epsilon}{2}.
}
On the other hand, by Young's inequality and uniform boundedness of $\nrm{\vec{u}_{\ast}(t)}_{H^{1} \times L^{2}}$, we have
$
\sup_{t \in [0, \infty)}\|Q_M u_{\ast}(t)\|_{L^\infty(\Hp^4)}\aleq_{M} 1.
$
Moreover, by Young's inequality and Proposition~\ref{prop:morawetz}, we also have
$
\|Q_M u_{\ast}\|_{L_{t,x}^4([0, \infty)\times\Hp^4)}\aleq 1.
$
By interpolation, we get
$
\|Q_M u_{\ast} \|_{L_{t,x}^5([0, \infty) \times\Hp^4)}\aleq_{M} 1.
$
Hence for $T > 0$ sufficiently large, we have
\ali{\label{small Strich 2}
\|Q_M u_{\ast}\|_{L_{t,x}^5([T,\infty)\times\Hp^4)}\leq\frac{\epsilon}{2}.
}
Our task now is to show that \eqref{small Strich 1} and \eqref{small Strich 2} together imply that $\vec{u}_{\ast}$ scatters forward in time, which is impossible by the pre-compactness of $K_{+}$ unless $\vec{u}_{\ast} \equiv 0$. 
We proceed by a minor variant of the proof of Proposition~\ref{prop:small-data} given in \cite[Proof of Proposition~5.3]{LOS1}.

On any interval $I$, we introduce the strengthened scattering norm 
\begin{equation*}
	\nrm{v}_{\tilde{S}(I)} 
	= 
	\nrm{v}_{L_{t}^{3} (I; L_{x}^{6}(\bbH^{4}))}
	+ \nrm{v}_{L_{t}^{\frac{5}{2}} (I; L_{x}^{\frac{20}{3}}(\bbH^{4}))}
	+ \nrm{v}_{L_{t}^{2} (I; L_{x}^{8}(\bbH^{4}))}.
\end{equation*}
By Duhamel's principle applied to \eqref{4d eq} and Strichartz estimates for \eqref{4d lin eq} (cf. Proposition 4.2 in \cite{LOS1}), for any interval $I \subseteq [T, \infty)$ we have 
\ant{
\|u_{\ast}\|_{\tilde{S}(I)}\lesssim 1+\|\NN(r,u_{\ast})\|_{N(I)},
}
where the implicit constant is independent of $I$ and $N(I)$ is defined in \eqref{eq:N-norm}. The nonlinearity $\NN$ satisfies the estimate (cf. \cite[Lemma~5.2]{LOS1})
\ant{
|\NN(r,u)|\lesssim e^{-r}|u|^2+|u|^3.
} 
We will estimate the $N(I)$ norms of the two terms on the RHS of this inequality separately. We simplify notation by writing $L_t^pL_x^q$ for $L_t^p(I; L_x^q(\Hp^4)).$ For the cubic term we have
\ant{
\||u_{\ast}|^3\|_{L_t^1L_x^2}&\leq\||u_{\ast}|^2 \abs{Q_M u_{\ast}} \|_{L_t^1L_x^2}+\||u_{\ast}|^2 \abs{(1-Q_M) u_{\ast}}\|_{L_t^1L_x^2}\\
&\leq \|u_{\ast}\|_{L_t^{5/2}L_x^{20/3}}^2\|Q_M u_{\ast}\|_{L_{t,x}^5}+\|u_{\ast}\|^2_{L_t^2L_x^8}\|(1-Q_M) u_{\ast}\|_{L_t^\infty L_x^4}\\
&\aleq \epsilon \|u_{\ast}\|^2_{\tilde{S}(I)},
}
by \eqref{small Strich 1} and \eqref{small Strich 2}. Similarly
\ant{
\|e^{-r}|u_{\ast}|^2\|_{L_t^{3/2}L_x^{12/7}}\leq\||u_{\ast}|^3\|_{L_t^1L_x^2}^{2/3}\|e^{-r}\|_{L_t^\infty L_x^4}\lesssim\epsilon^{2/3}\|u_{\ast}\|_{\tilde{S}(I)}^{4/3},
}
by the previous estimate. In sum, we have shown that
\begin{equation*}
	\nrm{u_{\ast}}_{\tilde{S}(I)} \aleq 1 + \eps \nrm{u_{\ast}}_{\tilde{S}(I)}^{2} + \eps^{2/3} \nrm{u_{\ast}}_{\tilde{S}(I)}^{4/3}
\end{equation*}
for every interval $I \subseteq [T, \infty)$. If $\epsilon$ is chosen sufficiently small, we deduce via a continuity argument that $\|u_{\ast}\|_{\tilde{S}[T,\infty)}<\infty$. By pre-compactness of $K_{+}$, it follows that $\vec{u}_{\ast} \equiv 0$, which is a contradiction. \qedhere 
\end{proof}

\bibliographystyle{plain}
\bibliography{researchbib}

 \bigskip
\bigskip

\centerline{\scshape Andrew Lawrie, Sung-Jin Oh}
\smallskip
{\footnotesize
% please put the address of the first author
 \centerline{Department of Mathematics, The University of California, Berkeley}
\centerline{970 Evans Hall \#3840, Berkeley, CA 94720, U.S.A.}
\centerline{\email{ alawrie@math.berkeley.edu, sjoh@math.berkeley.edu}}
} % Do not forget to end the {\footnotesize by the sign }

 \medskip

\centerline{\scshape Sohrab Shahshahani}
\medskip
{\footnotesize
% please put the address of the first author
 \centerline{Department of Mathematics, The University of Michigan}
\centerline{2074 East Hall, 530 Church Street
Ann Arbor, MI  48109-1043, U.S.A.}
\centerline{\email{shahshah@umich.edu}}
} % Do not forget to end the {\footnotesize by the sign }

\end{document}